\documentclass[a4paper, reqno, 11pt]{amsart}

\usepackage[usenames,dvipsnames]{color}
\usepackage{amsthm,amsfonts,amssymb,amsmath,amsxtra}
\usepackage{tikz}

\usepackage{tikz-cd}
\usepackage{todonotes,cancel}

\usepackage[all]{xy}
\SelectTips{cm}{}
\usepackage{xr-hyper}
\usepackage[colorlinks=false,
   citecolor=Black,
   linkcolor=Red,
   urlcolor=Blue]{hyperref}
\usepackage{verbatim}

\usepackage[margin=1.25in]{geometry}
\usepackage{mathrsfs}

\RequirePackage{xspace}
\RequirePackage{etoolbox}
\RequirePackage{varwidth}
\RequirePackage{enumitem}
\RequirePackage{tensor}
\RequirePackage{mathtools}
\RequirePackage{longtable}
\RequirePackage{multirow}

\def\ge{\geqslant}
\def\le{\leqslant}
\def\a{\alpha}

\def\D{\Delta}
\def\L{\Lambda}

\def\l{\lambda}

\def\i{^{-1}}

\def\<{\langle}
\def\>{\rangle}

\newcommand{\OCB}{\mathring{\CB}}

\newcommand{\BC}{\ensuremath{\mathbb {C}}\xspace}

\newcommand{{\BG}}{\ensuremath{\mathbb {G}}\xspace}

\newcommand{{\BK}}{\ensuremath{\mathbb {K}}\xspace}

\newcommand{\BN}{\ensuremath{\mathbb {N}}\xspace}

\newcommand{\BP}{\ensuremath{\mathbb {P}}\xspace}

\newcommand{\BR}{\ensuremath{\mathbb {R}}\xspace}

\newcommand{\BZ}{\ensuremath{\mathbb {Z}}\xspace}

\newcommand{\CB}{\ensuremath{\mathcal {B}}\xspace}

\newcommand{\CD}{\ensuremath{\mathcal {D}}\xspace}

\newcommand{\CX}{\ensuremath{\mathcal {X}}\xspace}

\newcommand{\CZ}{\ensuremath{\mathcal {Z}}\xspace}

\newcommand{\Rp}{\mathbb{R}_{>0}}
\newcommand{\ul}{\underline}

\def\tW{\tilde W}

\def\kk{\mathbf k}

\newtheorem{thm}{Theorem}
\newtheorem{theorem}{Theorem}
\newtheorem{prop}[thm]{Proposition}

\newtheorem{lem}[thm]{Lemma}

\newtheorem{cor}[thm]{Corollary}

\theoremstyle{definition}

\newtheorem{defi}[thm]{Definition}
\newtheorem{example}[thm]{Example}

\newtheorem{remark}[thm]{Remark}

\numberwithin{equation}{section}
\numberwithin{thm}{section}

\setitemize[0]{leftmargin=*,itemsep=\the\smallskipamount}
\setenumerate[0]{leftmargin=*,itemsep=\the\smallskipamount}

\renewcommand{\to}{%
   \ifbool{@display}{\longrightarrow}{\rightarrow}%
   }
\let\shortmapsto\mapsto
\renewcommand{\mapsto}{%
   \ifbool{@display}{\longmapsto}{\shortmapsto}%
   }
\newlength{\olen}
\newlength{\ulen}
\newlength{\xlen}
\newcommand{\xra}[2][]{%
   \ifbool{@display}%
      {\settowidth{\olen}{$\overset{#2}{\longrightarrow}$}%
       \settowidth{\ulen}{$\underset{#1}{\longrightarrow}$}%
       \settowidth{\xlen}{$\xrightarrow[#1]{#2}$}%
       \ifdimgreater{\olen}{\xlen}%
          {\underset{#1}{\overset{#2}{\longrightarrow}}}%
          {\ifdimgreater{\ulen}{\xlen}%
             {\underset{#1}{\overset{#2}{\longrightarrow}}}
             {\xrightarrow[#1]{#2}}}}%
      {\xrightarrow[#1]{#2}}
   }
\makeatother
\newcommand{\xyra}[2][]{%
   \settowidth{\xlen}{$\xrightarrow[#1]{#2}$}%
   \ifbool{@display}%
      {\settowidth{\olen}{$\overset{#2}{\longrightarrow}$}%
       \settowidth{\ulen}{$\underset{#1}{\longrightarrow}$}%
       \ifdimgreater{\olen}{\xlen}%
          {\mathrel{\xymatrix@M=.12ex@C=3.2ex{\ar[r]^-{#2}_-{#1} &}}}%
          {\ifdimgreater{\ulen}{\xlen}%
             {\mathrel{\xymatrix@M=.12ex@C=3.2ex{\ar[r]^-{#2}_-{#1} &}}}
             {\mathrel{\xymatrix@M=.12ex@C=\the\xlen{\ar[r]^-{#2}_-{#1} &}}}}}%
      {\mathrel{\xymatrix@M=.12ex@C=\the\xlen{\ar[r]^-{#2}_-{#1} &}}}%
   }
\makeatletter
\newcommand{\xla}[2][]{%
   \ifbool{@display}%
      {\settowidth{\olen}{$\overset{#2}{\longleftarrow}$}%
       \settowidth{\ulen}{$\underset{#1}{\longleftarrow}$}%
       \settowidth{\xlen}{$\xleftarrow[#1]{#2}$}%
       \ifdimgreater{\olen}{\xlen}%
          {\underset{#1}{\overset{#2}{\longleftarrow}}}%
          {\ifdimgreater{\ulen}{\xlen}%
             {\underset{#1}{\overset{#2}{\longleftarrow}}}
             {\xleftarrow[#1]{#2}}}}%
      {\xleftarrow[#1]{#2}}
   }
\newcommand{\isoarrow}{%
   \ifbool{@display}{\overset{\sim}{\longrightarrow}}{\xrightarrow\sim}%
   }

\begin{document}

\title[]{Total positivity in twisted product of flag varieties}
\author[Huanchen Bao]{Huanchen Bao}
\address{Department of Mathematics, National University of Singapore, Singapore.}
\email{huanchen@nus.edu.sg}

\author[Xuhua He]{Xuhua He}
\address{Department of Mathematics and New Cornerstone Science Laboratory, The University of Hong Kong, Pokfulam, Hong Kong, Hong Kong SAR, China}
\email{xuhuahe@hku.hk}


\begin{abstract}

We show that the totally nonnegative part of the twisted product of flag varieties of a Kac-Moody group admits a cellular decomposition, and the closure of each cell is a topological manifold with boundary. We also establish explicit parameterizations  of each totally positive cell.

In the special cases of double flag varieties and braid varieties, we show that the totally nonnegative parts are regular CW complexes homeomorphic to  closed balls. 
Moreover, we prove that the link of any totally nonnegative double Bruhat cell in a reductive group is a regular CW complex homeomorphic to a closed ball, solving an open problem of Fomin and Zelevinsky. 
\end{abstract}

\maketitle

\section{Introduction}
\subsection{Flag varieties} We first give a quick review of the totally nonnegative flag varieties. 

Let $G$ be a connected Kac-Moody group, split over $\BR$. Let $B^+=T U^+$ be a Borel subgroup of $G$ and let $B^-=T U^-$ be the opposite Borel subgroup. Let $W$ be the Weyl group of $G$. In \cite{Lus-1} and \cite{Lu-positive}, Lusztig introduced the totally nonnegative monoid $$G_{\ge 0}=U^+_{\ge 0} T_{>0} U^-_{\ge 0}=U^-_{\ge 0} T_{>0} U^+_{\ge 0}.$$

Let $\CB=G/B^+$ be the flag variety. It admits the decomposition into open Richardson varieties $\CB=\sqcup \mathring{\CB}_{v, w}$, where $v$ and $w$ are elements in the Weyl group $W$, and the open Richardson variety $ \mathring{\CB}_{v, w}$, by definition, is the intersection of the Schubert cell $B^+ \dot w   B^+/ B^+$ and the opposite Schubert cell $B^- \dot v   B^+/ B^+$. The totally nonnegative flag variety $\CB_{\ge 0}$ is defined as the closure of $U^-_{\ge 0}  B^+/ B^+$ in $\CB$ with respect to Hausdorff topology. 

The totally nonnegative flag is a ``remarkable polyhedral subspace'' (cf. \cite{Lus-1}). In particular, we have 
\begin{enumerate}
    \item (decomposition into cells) $\CB_{\ge 0}=\sqcup \CB_{v, w, >0}$, where each totally positive stratum $\CB_{v, w, >0}=\CB_{\ge 0} \cap  \mathring{\CB}_{v, w}$ is a cell and admits an explicit parametrization (see \cite{MR} and \cite{BH20});
    
    \item the closure of each totally positive stratum is a union of totally positive strata (see \cite{Ri06} and \cite{BH22}); 
    
    \item the closure of each totally positive stratum is homeomorphic to a closed ball (see \cite{GKL} and \cite{BH22});
    
    \item If $G$ is a connected reductive group, then the flag variety $\CB$ admits a natural duality which stabilizes the totally nonnegative part $\CB_{\ge 0}$ and sends each totally positive stratum to another totally positive stratum (see \cite{Lus-1} and \cite{Lu-2}). 
\end{enumerate}

\subsection{Main result} The main object in this paper is the twisted product of flag varieties $$\CZ=G \times^{B^+} G \times^{B^+} \cdots \times^{B^+} G/B^+ \quad (n \text{ factors}).$$

For any sequence $\ul{w} = (w_1, \cdots, w_n)$ of elements in $W$, we define $$\mathring{\CZ}_{\ul{w}}=B^+ \dot w_1 B^+ \times^{B^+} B^+ \dot w_2 B^+ \times^{B^+} \cdots \times^{B^+} B^+ \dot w_n B^+/B^+.$$ This is analogous to Schubert cell in the flag $\CB$. For any $v \in W$, we define  $\mathring{\CZ}^v$ as $m \i(B^- \dot v B^+/B^+)$, where $m: \CZ \to \CB$ is the convolution product. This is analogous to the opposite Schubert cell in the flag $\CB$. The open Richardson variety of $\CZ$ is defined by $\mathring{\CZ}_{v, \ul{w}}=\mathring{\CZ}_{\ul{w}} \cap \mathring{\CZ}^v$. These open Richardson varieties include as special cases the Bott-Samelson varieties and the braid varieties. 

We define the totally nonnnegative part $\CZ_{\ge 0}$ of $\CZ$ to be the Hausdorff clsoure of $(U^-_{\ge 0}, U^-_{\ge 0}, \ldots, U^-_{\ge 0} B^+/B^+)$ in $\CZ$. We set $ \CZ_{v, \ul{w}, >0}=\CZ_{\ge 0} \cap \mathring{\CZ}_{v, \ul{w}}$ and call it a {\it totally positive stratum } of $\CZ$. We have the decomposition $\CZ_{\ge 0}=\sqcup \CZ_{v, \ul{w}, >0}$.

The goal of this paper is to show that the totally nonnegative twisted product of flag is still a ``remarkable polyhedral subspace''. 

\begin{theorem}\label{thm:main} [Theorem~\ref{thm:Zp}, Proposition~\ref{prop:TM} $\&$ Proposition~\ref{prop:duality}]
(1) Each stratum $ \CZ_{v, \ul{w}, >0}$ is a topological cell and we have an explicit parameterization of the cell; 


(2) The closure of each totally positive stratum in $\CZ$ is a union of some totally positive strata in $\CZ$; 

(3) The closure of each totally positive stratum in $\CZ$ is a topological manifold with boundary, where the boundary consists of all lower dimensional strata. 

(4) If $G$ is a connected reductive group, we have a natural duality on $\CZ_{\ge 0}$ generalizing Lusztig's duality on totally nonnegative flag varieties.
\end{theorem}

 
\subsection{Special case: braid varieties}
A braid $\ul{s}$ is a sequence $(s_{i_1}, \dots, s_{i_n})$ of simple reflections in $W$. Let $w = s_{i_1} \ast \cdots \ast s_{i_n}$ be the Demazure product. The closed (resp. open) braid variety is defined as $\CZ_{w, \ul{s}}$ (resp. $\mathring{\CZ}_{w, \ul{s}}$). When $G$ is a reductive group, the open braid varieties contain open Richardson varieties as special cases. 

As a special case of Theorem \ref{thm:main}, we see that the totally nonnegative braid variety is a ``remarkable polyhedral subspace". We further establish the regularity theorem in this case.

\begin{theorem}[Theorem~\ref{thm:braid}]
The totally nonnegative braid variety $\CZ_{w, \ul{s}, \ge0 }$ is regular CW complex homeomorphic to a closed ball. 
\end{theorem}

The cluster algebra structure on the coordinate ring of the open braid variety $\mathring{\CZ}_{w, \ul{s}}$ for reductive groups has recently been established by Casals, Gorsky, Gorsky, Le, Shen, and Simental in \cite{CGG} and by Galashin, Lam,  Sherman-Bennett, and Speyer in \cite{GLBS, GLB}. It is interesting to compare the total positivity structure with the cluster algebra structure.

The totally nonnegative braid variety $\CZ_{w, \ul{s}, \ge0 }$ also provides a geometric realization  of the subword complex considered by Knutson and Miller in \cite[Corollary~3.8]{KM}.

\subsection{Special case: double flag varieties} In this subsection, we consider the twisted product of flag varieties with $2$ factors. In this case, $\CZ$ is isomorphic to the double flag $\CB \times \CB$. Under this identification, the open Richardson varieties in $\CZ$ are the intersections of the $B^+ \times B^-$-orbits and the diagonal $G$-orbits in $\CB \times \CB$. Below, we give an explicit example of totally nonnegative double flag varieties.  

\begin{example}
Let $G=SL_2$. In this case $\CB \cong \BP^1$ and $\CB_{\ge 0}=[0, \infty]$. The totally nonnegative double flag variety $\CZ_{\ge 0}=\{(a, b); a, b \in [0, \infty], b-a \ge 0\}$. So $\CZ_{\ge 0}$ is a proper subspace of $\CB_{\ge 0} \times \CB_{\ge 0}$ defined by some nontrivial inequalities.
\end{example}

The notion of totally nonnegative double flag varieties was introduced by Webster and Yakimov in \cite{WY}. The  totally nonnegative double flag variety of $SL_2$ above was the only case understood for many years. The essential difficulty is to handle the extra nontrivial inequalities arising from the double flag. There is also a representation-theoretical construction of $\CZ_{\ge 0}$ relying on the canonical bases on tensor products of $G$-modules. The difficulty in studying $\CZ_{\ge 0}$ versus $\CB_{\ge 0} \times \CB_{\ge 0}$ can also be seen from the difficulty in constructing (geometrically or categorically) canonical bases on tensor products of $G$-modules.

As a special case of Theorem \ref{thm:main}, we see that the totally nonnegative double flag variety is a ``remarkable polyhedral subspace''. This in particular verifies a conjecture of Webster and Yakimov in \cite{WY}. We further establish the regularity theorem for totally nonnegative double flag varieties. 

 \begin{theorem}[Theorem~\ref{thm:double}]
    The totally nonnegative double flag variety is a regular CW complex homeomorphic to a closed ball.
\end{theorem}

\subsection{Special case: double Bruhat cells} The double Bruhat cells of a reductive group are the intersection of the $B^+ \times B^+$-orbits and the $B^- \times B^-$-orbits on $G$. They play an important role in the theory of cluster algebras (see \cite{FZ}).

It is easy to see that the maximal torus $T$ of $G$ acts freely on each double Bruhat cell $G^{s, t} $ of $G$. We denote by $L^{s,t} = G^{s, t}  / T$ the reduced double Bruhat cell.  

We may identify the quotient space $G/T$ with a subvariety of $\CX$ sending a reduced double Bruhat cell $L^{s, t}$ to a stratum $\mathring{\CX}^u_{v, w}$ of the double flag variety. Moreover, under this identification, the image of the totally positive part $G^{s, t }_{>0}:=G^{s, t} \cap G_{\ge 0}$ of the double Bruhat cells $G^{s , t}$ coincides with the totally positive part $ {\CX}^u_{v, w, >0}$ of the stratum $\mathring{\CX}^u_{v, w}$ in the double flag variety $\CX$. We obtain 

\begin{theorem}[Theorem~ \ref{thm:link}]
The link of each totally positive double Bruhat cell of a connected reductive group is a regular CW complex homeomorphic to a closed ball. 
\end{theorem}


This solves an open problem of Fomin and Zelevinsky (see \cite[Page 9]{Pos} and \cite[Page 571]{GKL}). It also generalizes the work of Hersh \cite{Her} on the cells in $U^-_{\ge 0}$. The special case for $GL_n$ was first proved by Galashin, Karp and Lam in \cite{GKL} using the embedding of the double Bruhat cells in $GL_n$ into the Grassmannian $Gr(n, 2n)$ due to Postnikov and the regularity of postroid cells in $Gr(n, 2n)$. Our approach for $GL_n$ is different from \cite{GKL}.

\subsection{Our strategy} 


One key ingredient of our approach is the ``thickening'' map. Let $G$ be the Kac-Moody group we start with and $I$ be the set of simple roots. Let $\tilde I=I \sqcup \{\infty_1, \infty_2, \dots, \infty_{n-1}\}$ be the thickening of $I$ obtained by adding $n-1$ vertices labeled with $\infty_1$, $\dots$, $\infty_{n-1}$ and the edges $\Leftrightarrow$ between the vertices $\infty_i$ and the vertices in $I$ (see Figure \ref{fig:1} for $G= SL_4$ and $ n =3$). Let $\tilde G$ be the corresponding ``thickening'' Kac-Moody group.

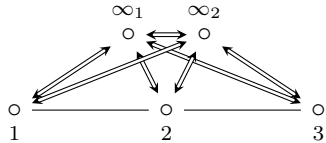
\begin{figure}[h]
\centering
\begin{tikzpicture}[baseline=0]
	\node  (a) at (0,0) {$\circ$};
	\node   at (0,-0.3) {$\scriptstyle 1$};
	\node (b) at (2,0) {$\circ$};
	\node at (2, -0.3) {$\scriptstyle 2$};
	\node (c) at (4,0) {$\circ$};
	\node at (4, -0.3) {$\scriptstyle 3$};
	\draw (a) -- (b) -- (c);
	\node (d) at (1.5,1) {$\circ$};
	\node at (1.5,1.3) {$\scriptstyle \infty_1$};
		\node (e) at (2.5,1) {$\circ$};
	\node at (2.5,1.3) {$\scriptstyle \infty_2$};
	\draw[>=stealth, double, double distance=1pt, <->] (a) --  (d) ;
	\draw[>=stealth, double, double distance=1pt, <->] (b) -- (d);
	\draw[>=stealth, double, double distance=1pt, <->] (c) -- (d);
		\draw[>=stealth, double, double distance=1pt, <->] (a) --  (e) ;
	\draw[>=stealth, double, double distance=1pt, <->] (b) -- (e);
	\draw[>=stealth, double, double distance=1pt, <->] (c) -- (e);
	\draw[>=stealth, double, double distance=1pt, <->] (d) -- (e);
	\end{tikzpicture}
\caption{$\tilde G$ for $SL_4$}\label{fig:1}
\end{figure}

We construct in Proposition \ref{prop:Z} a locally trivial $(\BC^*)^{n-1}$-fibration from a subvariety $\CZ' \subset \tilde{\CB}$ to $\CZ$. The projection map sends an open Richardson variety in $\tilde{\CB}$ to a stratum $\mathring{\CZ}_{v, \ul{w}}$ in $\CZ$. Moreover, the fibration is compatible with the totally positive structure on the twisted product of flag varieties $\CZ$ and on the thickening flag variety $\tilde \CB$. We then deduce the local structure on the totally nonnegative twisted product of flag varieties $\CZ_{\ge 0}$ from the local structure on the totally nonnegative flag variety $\tilde \CB$ established in \cite{BH22}.  

It is worth mentioning that even for the totally nonnegative double flag varieties of reductive groups, our proof still relies on the knowledge of the totally nonnegative flag of Kac-Moody groups (of non-finite, non-affine type) in \cite{BH22}.

\vspace{.2cm}
\noindent {\bf Acknowledgement: } HB is supported by MOE grants A-0004586-00-00 and A-0004586-01-00. XH is partially supported by the New Cornerstone Foundation through the New Cornerstone Investigator Program and the Xplorer Prize, and by Hong Kong RGC Grants 14300021 and 14300023. We thank George Lusztig for helpful discussions and for the question on \S\ref{subsec:Lusztig}, and we thank Steven Karp for informing us of the background on the open problem of Fomin and Zelevinsky. We thank Lauren Williams for helpful discussions on shellability of posets.
 
\section{Flag varieties and their totally nonnegative parts}

\subsection{Kac-Moody groups} 
A {\it Kac-Moody root datum} is a quintuple $$\CD=(I, A, X, Y, (\a_i)_{i \in I}, (\a^\vee_i)_{i \in I}),$$ where $I$ is a finite set, $A=(a_{ij})_{i, j \in I}$ is a symmetrizable generalized Cartan matrix in the sense of \cite[\S 1.1]{Kac}, $X$ is a free $\BZ$-module of finite rank with $\BZ$-dual $Y$, and the elements $\a_i$ of $X$ and $\a^\vee_i$ of $Y$ such that $\<\a^\vee_j, \a_i\>=a_{ij}$ for $i, j \in I$. 

Let $\Delta^{re} = \{w ( \pm \alpha_i) \in X \mid i \in I, w \in W\} \subset X$ be the set of real roots. Then $\D^{re}=\D^{re}_+ \sqcup \D^{re}_-$ is the union of positive real roots and negative real roots. Let $\kk$ be an algebraically closed field. The {\it minimal Kac-Moody group} $G$ associated with the Kac-Moody root datum $\CD$ is the group generated by the torus $T=Y \otimes_\BZ \kk^\times$ and the root subgroup $U_\a \cong \kk$ for each real root $\a$, subject to the Tits relations \cite{Ti87}. Let $U^+  \subset G $ (resp. $U^-  \subset G$) be the subgroup generated by $U_\a$ for $\a \in \D^{re}_+$ (resp. $\a \in \D^{re}_-$). Let $B^{\pm}  \subset G $ be the Borel subgroup generated by $T$ and $U^{\pm}$. 
For each $i \in I$, we fix isomorphisms $x_i: \BC \to U_{\a_i}$ and $y_i: \BC \to U_{-\a_i}$ such that the map
\[
\begin{pmatrix} 1 & a  \\ 0 & 1 \end{pmatrix} \mapsto x_i(a), \begin{pmatrix} b & 0  \\ 0 & b \i \end{pmatrix} \mapsto \a^\vee_i(a), \begin{pmatrix} 1 & 0  \\ c & 1 \end{pmatrix} \mapsto y_i(c)
\] defines a homomorphism $SL_2 \to G$. The data $(T, B^+, B^-, x_i, y_i; i \in I)$ is called a {\it pinning} for $G$.

Let $W$ be the Weyl group of $G$. For $i \in I$, we denote by $s_i \in W$ the corresponding simple reflection and write $\dot{s}_i = x_{i}( 1) y_i (-1) x_{i}(1) \in G$. For any $w \in W$ with reduced expression $w = s_{i_1} \cdots s_{i_n}$, we define $\dot{w} = \dot{s}_{i_1} \cdots \dot{s}_{i_n} \in G$. It is known that $\dot{w}$ is independent of the choice of the reduced expressions.

Let $J \subset I$ (not necessarily of finite type). We denote by $P^+_J$ the standard parabolic subgroup of $G$ corresponding to $J$ and $P^-_J$ the opposite parabolic subgroup of $G$. Let $L_J=P^+_J \cap P^-_J$ be the standard Levi subgroup. Let $W_J$ be the subgroup of $W$ generated by $\{s_j\}_{j \in J}$. Then $W_J$ is the Wey group of $L_J$. For any parabolic subgroup $P$, we denote by $U_P$ its unipotent radical. We have the Levi decomposition $P^\pm_J=L_J \ltimes U_{P^\pm_J}$. 



\subsection{Flag varieties}\label{sec:flag}
We denote by $\CB=G/B^+$ the (thin) full flag variety \cite{Kum}, equipped with the ind-variety structure.  Let $v, w \in W$. Let $\mathring{\CB}_w =B^+ \dot{w} B^+/B^+$ be the {\it Schubert cell} corresponding to $w$, $\mathring{\CB}^v=B^- \dot{v} B^+/B^+$ be the {\it opposite Schubert cell} corresponding to $v$ and $\mathring{\CB}_{v, w}=\mathring{\CB}_w \cap \mathring{\CB}^v$ be the {\it open Richardson variety} corresponding to the pair $(v, w)$. It follows from \cite{Kum1} that 

(a) {\it $\mathring{\CB}_{v, w} \neq \emptyset$ if and only if $v \le w$. In this case, $\mathring{\CB}_{v, w}$ is irreducible of dimension $\ell(w)-\ell(v)$.}

We have the decompositions $\CB=\sqcup_{w \in W} \mathring{\CB}_w=\sqcup_{v \in W} \mathring{\CB}^v=\sqcup_{v \le w} \mathring{\CB}_{v, w}$. Let $\CB_w, \CB^v, \CB_{v, w}$ be the (Zariski) closure of $\mathring{\CB}_w, \mathring{\CB}^v, \mathring{\CB}_{v, w}$ respectively. By \cite[Proposition~7.1.15 \& 7.1.21]{Kum}, we have
$$
\CB_w=\bigsqcup_{w' \le w} \mathring{\CB}_{w'}, \qquad \CB^v=\bigsqcup_{v' \ge v} \mathring{\CB}^{v'}, \qquad \CB_{v, w}=\bigsqcup_{v \le v' \le w' \le w} \mathring{\CB}_{v', w'}.
$$

\subsection{Expressions and subexpressions}\label{subssec:posiex}
An {\it expression} is a sequence 
$$
\mathbf w=(t_1, t_2, \ldots, t_n)
$$
in $I \sqcup \{1\}$. In this case, we write $w=t_1 t_2 \cdots t_n \in W$, and we say that $\mathbf w$ is an expression of $w$. The expression $\mathbf w$ is called {\it reduced} if $n=\ell(w)$. In this case, $t_i \in I$ for all $i$. 

A subexpression of $\mathbf w$ is a sequence $\mathbf w'=(t'_1, t'_2, \ldots, t'_n)$, where $t'_i \in \{t_i, 1\}$ for all $i$. If $\mathbf w$ is an reduced expression, then the subexpression $\mathbf w'$ is called {\it positive} if $t'_1 \cdots t'_{i-1}<t'_1 \cdots t'_{i-1} t_i$ for all $i$. 

By \cite[Lemma 3.5]{MR}, for any $w, w' \in W$ with $w' \le w$ and for any reduced expression $\mathbf w$ of $w$, there exists a unique positive subexpression for $w'$ in $\mathbf w$. We denote it by $\mathbf w'_{+}$.

\subsection{Totally nonnegative flag varieties}\label{Rp-BH} 
We follow \cite{Lus-1} and \cite{Lu-2}. Let $U^+_{\ge 0}$ be the submonoid of $G$ generated by $x_i(a)$ for $i \in I$ and $a \in \Rp$ and $U^-_{\ge 0}$ be the submonoid of $G$ generated by $y_i(a)$ for $i \in I$ and $a \in \Rp$. Let $T_{>0}$ be the identity component of $T(\BR)$. The totally nonnegative monoid $G_{\ge 0}$ is defined to be the submonoid of $G$ generated by $U^{\pm}_{\ge 0}$ and $T_{>0}$. By \cite[\S 2.5]{Lu-2}, $G_{\ge 0}=U^+_{\ge 0} T_{>0} U^-_{\ge 0}=U^-_{\ge 0} T_{>0} U^+_{\ge 0}$.

Let $\CB_{\ge 0}=\overline{U^-_{\ge 0} B^+/ B^+}$ be the closure of $U^-_{\ge 0} B^+/ B^+$ in $\CB$ with respect to the Hausdorff topology. For any $v \le w$, let $
\CB_{v, w, >0}=\mathring{\CB}_{v, w} \bigcap \CB_{\ge 0}$. Then $\CB_{\ge 0}=\bigsqcup_{v \le w} \CB_{v, w, >0}$. 


Following \cite[Definition~5.1]{MR}, we set 
\[
 G_{\bf v_+, \bf w, >0}=\{g_1 g_2 \cdots g_n\vert  g_j=\dot s_{i_j}, \text{ if } t_{i_j}=s_{i_j}; \text{ and } g_j \in y_{i_j}(\Rp), \text{ if } t_{i_j}=1\}.
\]
Note that the obvious map $\Rp^{\ell(w) - \ell(v)} \rightarrow  G_{\bf v_+, \bf w, >0}$
is a homeomorphism.

By \cite[Theorem~11.3]{MR} for the reductive groups and \cite[Theorem~4.10]{BH20} for the Kac-Moody groups, we have the following result. 

(a) {\it Let $v \le w$. For any reduced expression ${\bf w}$ of $w$, the map $g \mapsto g \cdot B^+$ gives a homeomorphism $G_{\bf v_+, \bf w, >0} \cong \CB_{v, w, >0}$.
In particular, $\CB_{v, w, >0} \cong \Rp^{\ell(w)-\ell(v)}$ is a topological cell.}

\begin{prop}\label{thm:closure}\cite[Proposition~5.2]{BH22}
        Let $v \le w$. Then 
        
        (1) $\CB_{v, w, >0}$ is a connected component of $\CB_{v, w}(\BR)$. 
        
        (2) The Hausdorff closure of $\CB_{v, w, >0}$ equals $\sqcup_{v \le v' \le w' \le w} \CB_{v', w', >0}$. 
\end{prop}
For reductive groups, the above results were established by Rietsch in \cite{Ri99, Ri06}.

\subsection{Twisted product of flag varieties}
Let $n \in \BN$. Recall the twisted product $\CZ_n$ of the flag varieties: $$\CZ_n=G \times^{B^+} G \times^{B^+} \cdots \times^{B^+} G/B^+ \quad (n \text{ factors}).$$

We often write $\CZ = \CZ_n$ whenever the number of factors is clear. The isomorphism $G^n \rightarrow G^n$, $(g_1, g_2, \dots, g_n) \mapsto (g_1, g_1g_2, \dots, g_1g_2\cdots g_n)$ induces an isomorphism $$\a: \CZ_n \to \CB^n, \qquad (g_1, g_2, \ldots, g_n) \mapsto (g_1 B^+, g_1 g_2 B^+, \ldots, g_1 g_2 \cdots g_n B^+).$$ In particular, we have the convolution product $$m: \CZ_n \to \CZ_1=G/B^+, \qquad  (g_1, \ldots, g_n) \mapsto g_1 g_2 \cdots g_n B^+.$$ 

The stratification on $\CZ_n$ we are interested in is different from the stratification on $\CB^n$ given by the product of $\mathring{\CB}_{-,-}$. Let $\ul{w} = (w_1, \cdots, w_n)$ be a sequence of elements in $W$. Recall that the Schubert cell in $\CZ_n$ associated with $\ul{w}$ is defined to be $$\mathring{\CZ}_{n, \ul{w}}=B^+ \dot w_1 B^+ \times^{B^+} B^+ \dot w_2 B^+ \times^{B^+} \cdots \times^{B^+} B^+ \dot w_n B^+/B^+.$$ 

Also, recall that the opposite Schubert cell in $\CZ_n$ associated with $v \in W$ is defined to be $$\mathring{\CZ}^v_n=\{x \in \CZ_n; m(x) \in \mathring{\CB}^v\}.$$

The open Richardson variety of $\CZ_n$ is defined by $\mathring{\CZ}_{n, v, \ul{w}}=\mathring{\CZ}_{n, \ul{w}} \cap \mathring{\CZ}^v_n$. 

\subsection{Monoid structure on $W$} 
We follow \cite[\S 1]{He-min}. For any $x, y \in W$, the subset $\{x y'; y' \le y\}$ (resp. $\{x' y; x' \le x\}$, $\{x' y'; x' \le x, y' \le y\}$) of $W$ contains a unique maximal element (with respect to the Bruhat order $\le$). Moreover, we have $$\max\{x y'; y' \le y\}=\max\{x' y; x' \le x\}=\max\{x' y'; x' \le x, y' \le y\}.$$ We denote this element by $x \ast y$ and call it the {\it Demazure product} of $x$ and $y$. Moreover, $(W, \ast)$ is a monoid and the Demazure product can be determined by the following two rules
\begin{itemize}
    \item $x \ast y = x y$ if $x, y \in W$ such that $\ell(x y) = \ell(x) + \ell(y)$;
    \item $s \ast w = w$ if $s \in I$, $w \in \tW$ such that $s w < w$.
\end{itemize}


We define the maps $m_\bullet, m_\ast: W^n \to W$ by \begin{gather*} 
m_{\bullet}(w_1, w_2, \ldots, w_n)=w_1 w_2 \cdots w_n, \\ 
m_{\ast}(w_1, w_2, \ldots, w_n)=w_1 \ast w_2 \ast \cdots \ast w_n.
\end{gather*}

We have the following result describing the nonemptiness pattern of the Richardson varieties in $\CZ_n$. 

\begin{lem}\label{lem:nonempty}
Let $v \in W$ and $\ul{w} \in W^n$. Then $\mathring{\CZ}_{n,v, \ul{w}} \neq \emptyset$ if and only if $v \le m_{\ast}(\ul{w})$.
\end{lem}

\begin{proof}
We have $\mathring{\CB}_{m_{\ast}(\ul{w})} \subset m(\mathring{\CZ}_{n, \ul{w}}) \subset \sqcup_{w' \le m_{\ast}(\ul{w})} \mathring{\CB}_{w'}$. Now the statement follows from the equivalence of the following three conditions: 
\[
	v \le m_{\ast}(\ul{w}) \quad\Longleftrightarrow \quad\mathring{\CB}^v \cap \mathring{\CB}_{m_{\ast}(\ul{w})} \neq \emptyset  \quad \Longleftrightarrow \quad\mathring{\CB}^v \cap (\sqcup_{w' \le m_{\ast}(\ul{w})} \mathring{\CB}_{w'}) \neq \emptyset. \qedhere
\]
\end{proof}



\section{The thickening  map}
 Let $ n \in \BZ_{> 0}$. We write $\CZ = \CZ_n$ throughout this section.
\subsection{The thickening group}
We construct the set of simple roots and the associated generalized Cartan matrix of the thickening Kac-Moody $\tilde{G}$ from the original Kac-Moody group ${G}$. 

The set $\tilde{I}$ of simple roots is the union of $I$ with new vertices $\{\infty_1, \ldots, \infty_{n-1}\}$. The Dynkin diagram of $\tilde I$ is obtained from the Dynkin diagram of $I$ by adding an edge $i \Longleftrightarrow \infty_l$ for any $i \in I$ and $1 \le l \le n-1$. In other words, the generalized 
Cartan matrix $\tilde A=(\tilde a_{i, j})_{i, j \in \tilde I}$ is defined as follows:\footnote{In fact, the argument in this paper works once $\tilde a_{i, \infty_l}, \tilde a_{\infty_l, i} \neq 0$ for $i \in I$ and $1 \le l \le n-1$. } 
\begin{itemize}
	\item $\tilde a_{i, j}=a_{i, j}$ for $i, j \in I$; 
	\item $\tilde a_{i, \infty_l}=\tilde a_{\infty_l, i}=-2$ for $i \in I$ and $1 \le l \le n-1$;  
	\item $\tilde a_{\infty_l, \infty_l}=2$ for $1 \le l \le n-1$;
	$\tilde a_{\infty_l, \infty_{l'}}=-2$ for $1 \le l, l' \le n-1$ with $l \neq l'$. 
\end{itemize}

Let $\tilde G$ be the minimal Kac-Moody group of simply connected type associated to $(\tilde I, \tilde A)$ and $\tilde W$ be its Weyl group. Let $\tilde W_I$ be the parabolic subgroup of $\tilde W$ generated by simple reflections in $I$ and $\tilde L_I$ be the standard Levi subgroup of the parabolic subgroup $\tilde P^+_I$ of $\tilde G$ that corresponds to $I$. We have natural identifications 
\begin{gather*}
    i: W \to \tilde W_I \text{ and } i: G \to \tilde L_I.
\end{gather*}

We fix a pinning $(\tilde{T}, \tilde{B}^+ , \tilde{B}^-, x_i, y_i; i \in \tilde{I})$ of $\tilde{G}$ that is compatible with the pinning of $G$ via the identification $i$. We define the thickening maps 
\begin{gather*}
    th: W^n \to \tilde W, (w_1, w_2, \ldots, w_n) \mapsto w_1 s_{\infty_1} w_2 s_{\infty_2} \cdots s_{\infty_{n-1}} w_n; \\
    th: (\BC^*)^{n-1} \times G^n  \to \tilde G, (a_1, \dots, a_{n-1} , g_1, \ldots, g_n) \mapsto g_1 y_{\infty_1}(a_1) g_2 \cdots g_{n-1}y_{\infty_{n-1}}(a_{n-1}) g_n.
\end{gather*}

We denote by $\tilde \CB$ the flag variety of $\tilde G$, and denote by $\mathring{\tilde \CB}^-$, $ \mathring{\tilde \CB}_-$, $\mathring{\tilde \CB}_{-, -}$ the Schubert cells, the opposite Schubert cells and the open Richardson varieties of $\tilde \CB$ respectively.

\subsection{Positive subexpressions in $W^n$}\label{subsec:posi}
Let $\ul{w} = (w_1, \dots, w_n)$ and $\underline{w'} = (w'_1, \dots, w'_n)$ be in $W^n$. We say $\underline{w'}  \le \underline{w}$ if $w'_i \le w_i$ in the usual Bruhat order for all $i$.

\begin{defi} Let $\ul{w}, \ul{v} \in W^n$ with $\ul{v} \le \ul{w}$. We say that $\ul{v}$ is positive in $\ul{w}$ if for any $i$, $v_1' \le v_1, \dots, v'_{i-1} \le v_{i-1}, v'_i \le w_i$ with $v_1' \cdots v_{i-1}' v'_{i} =  v_1  \cdots v_{i-1} v _{i}$ implies that $(v_1' , \cdots, v_i') =(v_1 , \cdots, v_i)$. 
\end{defi}

Let $\ul{w} = (w_1, \dots, w_n) \in W^n$. We choose an expression $\mathbf w_i$ for each $w_i$. We write $\ul{\mathbf w}=(\mathbf w_1; \mathbf w_2; \dots ; \mathbf w_n)$. We then associate  two expressions to it: 
\begin{itemize}
    \item $i(\ul{\mathbf w})=(\mathbf w_1, 1, \mathbf w_2, 1, \dots, 1, \mathbf w_n)$, an expression of $i(m_\bullet(\ul{w}))=w_1 w_2 \cdots w_n$; 
    \item $th(\ul{\mathbf w})=(\mathbf w_1, s_{\infty_1}, \mathbf w_2, \dots,s_{\infty_{n-1}}, \mathbf w_n  )$, an expression of $th(\ul{w})=w_1 s_{\infty_1} w_2 \cdots s_{\infty_{n-1}} w_{n}$. 
\end{itemize}

It is easy to see that 
\begin{itemize}
    \item for $v, v' \in W$, $v \le v'$ in $W$ if and only if $i(v) \le i(v')$ in $\tilde W$; 
    \item for $\ul{w}, \ul{w'} \in W^n$, $\ul{w} \le \ul{w'}$ if and only if $th(\ul{w}) \le th(\ul{w'})$ in $\tilde W$; 
    \item for $v \in W$ and $\ul{w} \in W^n$, $v \le m_\ast(\ul{w})$ if and only if $i(v) \le th(\ul{w})$.
\end{itemize}  

The following result follow  directly from the definition. 

   (a) {\it Let $\ul{w}, \ul{v} \in W^n$ with $\ul{v} \le \ul{w}$. Let $\mathbf w_i$ be a reduced expression of $w_i$ and $(\mathbf v_i)_+$ be the positive subexpression for $v_i$ in $\mathbf w_i$. Then $\ul{v}$ is positive in $\ul{w}$ if and only if $i(\ul{\mathbf v}_+)$ is a positive subexpression of $th(\ul{\mathbf w})$, where $\ul{\mathbf v}_+=((\mathbf v_1)_+; (\mathbf v_2)_+; \dots; (\mathbf v_n)_+)$ and $\ul{\mathbf w}=(\mathbf w_1;   \mathbf w_2; \dots; \mathbf w_n)$} in the sense of \S\ref{subssec:posiex}.

   (b) {\it Let $\ul{w} \in W^n$ and $v \in W$ with $v \le m_\ast(\ul{w})$. Then there exists a unique $\ul{v}$ in $W^n$ with $m_{\bullet}(\ul{v})=v$ and $\ul{v}$ is positive in $\ul{w}$.  In this case, $m_{\ast}(\ul{v})=v$.}

\subsection{The thickening map}
Let $\alpha_{\infty_l}$ be the simple root corresponding to $\infty_l \in \tilde{I}$ for $ 1\le l \le n-1$. We write the subvariety $U_{-\a_{\infty_l}} - \{e\} \subset U_{-\a_{\infty_l}}$ as $U_{-\a_{\infty_l}}^*$.  
\begin{prop}\label{prop:Z}
 Set  $$\CZ'=\tilde P^-_{I} \cdot \tilde B^+/\tilde B^+ \cap \tilde P^+_I (\tilde B^+ \dot s_{\infty_1} \tilde B^+) \tilde P^+_I (\tilde B^+ \dot s_{\infty_2} \tilde B^+) \cdots (\tilde B^+ \dot s_{\infty_{n-1}} \tilde B^+) \tilde P^+_I/\tilde B^+.$$ 
 Then 
 \begin{enumerate}
 \item $\CZ'=  \bigsqcup_{v \le m_\ast(\ul{w})}  \mathring{\tilde{\CB}}_{i(v), th(\ul{w})}$ is a locally closed subvariety of $\tilde \CB$; 
 \item there exists a locally trivial $(\BC^*)^{n-1}$-fibration $\pi: \CZ' \to \CZ$;
 \item the morphism $\pi$ is stratified such that $\pi$ maps $\mathring{\tilde{\CB}}_{i(v), th(\ul{w})}$to $\mathring{\CZ}_{v, \ul{w}}$ via restriction. 
 \end{enumerate}
 \end{prop}
 
 \begin{proof}
 Set $\CZ''=\tilde P^+_I (\tilde P^+_{\infty_1}) \tilde P^+_I (\tilde P^+_{\infty_2})\cdots (\tilde P^+_{\infty_{n-1}}) \tilde P^+_I/\tilde B^+$. Then $\CZ''$ is closed in $\tilde \CB$. Moreover, $\tilde P^+_I (\tilde B^+ \dot s_{\infty_1} \tilde B^+) \tilde P^+_I (\tilde B^+ \dot s_{\infty_2} \tilde B^+) \cdots (\tilde B^+ \dot s_{\infty_{n-1}} \tilde B^+) \tilde P^+_I/\tilde B^+$ is open in $\CZ''$. As $\tilde P^-_{I} \cdot \tilde B^+/\tilde B^+$ is open in $\tilde \CB$, $\CZ'$ is locally closed in $\tilde \CB$. This shows (1).

 We have the following  isomorphism via the convolution product 
\begin{align*}
	&G \times ^{B^+} (\tilde B^+ \dot s_{\infty_1} \tilde B^+) \times^{\tilde B^+} G  \times^{B^+} (\tilde B^+ \dot s_{\infty_2} \tilde B^+)  \times^{\tilde B^+} \cdots  \times^{ B^+} (\tilde B^+ \dot s_{\infty_{n-1}} \tilde B^+)  \times^{\tilde B^+} \tilde P^+_I/\tilde B^+ \\
	\cong \,&  G (\tilde B^+ \dot s_{\infty_1} \tilde B^+)   G (\tilde B^+ \dot s_{\infty_2} \tilde B^+) \cdots (\tilde B^+ \dot s_{\infty_{n-1}} \tilde B^+)   G/\tilde B^+\\
	\cong \,& \tilde P^+_I (\tilde B^+ \dot s_{\infty_1} \tilde B^+) \tilde P^+_I (\tilde B^+ \dot s_{\infty_2} \tilde B^+) \cdots (\tilde B^+ \dot s_{\infty_{n-1}} \tilde B^+) \tilde P^+_I/\tilde B^+.
\end{align*}

Here the $B^+$-action on $\tilde B^+ \dot s_{\infty_1} \tilde B^+$ is by multiplication, while the action of $\tilde B^+$ on $G$ is via the quotient $\tilde B^+ \to B^+$.  
Any element of $\CZ'$ is of the form $g_1 y_1 g_2 y_2 \cdots y_{n-1} g_n \tilde B^+/\tilde B^+$, where $g_i \in G$ and $y_i \in \tilde {B}^+ \dot  s_{\infty_i} \tilde B^+$. It follows from $\CZ' \subset \tilde P^-_{I} \cdot \tilde B^+/\tilde B^+$ that $y_i \in \tilde {B}^+ \dot  s_{\infty_i} \tilde B^+ \cap \tilde{B}^-\tilde B^+$. 
Hence we obtain, via restriction, 
\begin{align*}
	Z' \cong & G \times ^{ B^+}  (U_{-\a_{\infty_1}}^* \tilde B^+) \times^{\tilde B^+} G  \times^{  B^+} ( U_{-\a_{\infty_2}}^*  \tilde B^+)  \times^{\tilde B^+} \cdots   \times^{\tilde B^+} \tilde P^+_I/\tilde B^+ \\
	  \cong &G \times ^{ B^+} \big( U_{-\a_{\infty_1}}^* \times G \big)  \times^{  B^+} ( U_{-\a_{\infty_2}}^*  \tilde B^+)  \times^{\tilde B^+} \cdots \cdots   \times^{\tilde B^+} \tilde P^+_I/\tilde B^+.
\end{align*}
The $B^+$-action on $G \times \big( U_{-\a_{\infty_1}}^* \times G)$ is given by $b \cdot (x,y, z) = (xb^{-1}, byb^{-1}, bz)$. Similar $B^+$-actions apply for the other factors. We hence have a well-defined morphism 
\begin{equation}\label{eq:pi1}
	 \pi_1: \CZ' \rightarrow G \times ^{ B^+}  G   \times^{  B^+} ( U_{-\a_{\infty_2}}^*  \tilde B^+)  \times^{\tilde B^+} \cdots \cdots   \times^{\tilde B^+} \tilde P^+_I/\tilde B^+.
\end{equation}
Iterating the construction, we obtain a morphism 
\[
	 \pi: \CZ' \rightarrow G \times ^{ B^+}  G   \times^{  B^+}    \cdots   \times^{ B^+} \tilde P^+_I/\tilde B^+ \cong \CZ.
\]
Part (3) follows immediately. 

We prove (2). Let $r \in W$. We have
\begin{align*}
	\CZ' &\supset (\dot{r}U^- B^+) \times ^{ B^+} \big( U_{-\a_{\infty_1}}^* \times G \big)  \times^{  B^+} ( U_{-\a_{\infty_2}}^*  \tilde B^+)  \times^{\tilde B^+} \cdots \cdots   \times^{\tilde B^+} \tilde P^+_I/\tilde B^+  \\
	&\cong \dot{r}U^-   \times   U_{-\a_{\infty_1}}^* \times G   \times^{  B^+} ( U_{-\a_{\infty_2}}^*  \tilde B^+)  \times^{\tilde B^+} \cdots \cdots   \times^{\tilde B^+} \tilde P^+_I/\tilde B^+ \\
	& \cong U_{-\a_{\infty_1}}^* \times   (\dot{r}U^- B^+) \times ^{ B^+}    G   \times^{  B^+} ( U_{-\a_{\infty_2}}^*  \tilde B^+)  \times^{\tilde B^+} \cdots \cdots   \times^{\tilde B^+} \tilde P^+_I/\tilde B^+.
\end{align*}
This shows that $\pi_1$ in \eqref{eq:pi1} is a locally trivial $\BC^*$-fibration. Part (2) follows by iteration. 
 \end{proof}
 
For any $\ul{r} = (r_1, \dots, r_n) \in W^n$, the thickening map $th: (\BC^*)^{n-1} \times G^n \to \tilde G$ induces a section of the fiber bundle $\pi: \CZ' \rightarrow \CZ$ as follows 
 \begin{align}\label{eq:th}
	  \,\, \BC^* \times ( \dot{r}_1 U^-) \times  \BC^* \times ( \dot{r}_2 U^-) \times \cdots \times (\dot{r}_n U^-)/B^+ \rightarrow  \CZ'. 
 \end{align}
 
 
 The projection $\pi$ in Proposition~\ref{prop:Z} is defined over $\mathbb{R}$, hence induces projections on the set of $\mathbb{R}$-rational points 
\begin{equation}\label{eq:Z}
 \pi: \CZ'(\mathbb{R}) \rightarrow \CZ (\mathbb{R}),  \qquad  \pi: \mathring{\tilde{\CB}}_{i(v), th(\ul{w})}(\BR) \rightarrow \mathring{\CZ}_{v, \ul{w}}(\BR).
\end{equation}

We draw some consequences from Proposition~\ref{prop:Z}.
\begin{cor}\label{cor:irre}
Let $v \in W$ and $\ul{w} = (w_1, w_2, \dots, w_n) \in W^n$ with $v \le m_{\ast}(\ul{w})$. Then
\begin{enumerate}
    \item $\mathring{\CZ}_{v, \ul{w}}$ is irreducible and of dimension $\ell(w_1) +\ell(w_2) + \dots + \ell(w_n)-\ell(v)$;
    \item the Zariski closure of $\mathring{\CZ}_{v, \ul{w}}$ equals $\CZ^v\cap \CZ_{ \ul{w}}=\bigsqcup_{v \le v', \ul{w'} \le \ul{w} \text{ with } v' \le m_{\ast}(\ul{w'})} \mathring{\CZ}_{v',  \ul{w'}}$. 
\end{enumerate}
\end{cor}

We may also deduce Lemma \ref{lem:nonempty} from Proposition \ref{prop:Z} and the nonemptiness pattern of the open Richardson varieties in $\tilde \CB$. 

\subsection{Total positivity on $\CZ $}

Let $\CZ_{ \ge 0}$ be the Hausdorff closure of $(U^-_{\ge 0}, U^-_{\ge 0},  \cdots,  U^-_{\ge 0}B^+/B^+)$ in $\CZ $. For any $v \in W$ and $\ul{w} \in W^n$ with $v \le m_\ast(\ul{w})$, we define
$$\CZ_{v, \ul{w}, >0}=\mathring{\CZ}_{v,  \ul{w}} \cap \CZ_{\ge 0}.$$ 

Now we prove the first main result of this paper. 
    
\begin{thm}\label{thm:Zp}
    Let $v \in W$ and $\ul{w} \in W^ n $ with $v \le  m_\ast(\ul{w})$. Then 

    (1) $\CZ_{v,  \ul{w}, >0}$ is a connected component of $\mathring{\CZ}_{v, \ul{w}}(\BR)$. 
    
    (2) The Hausdorff closure of $\CZ_{v,  \ul{w}, >0}$ equals $\bigsqcup_{v \le v', \ul{w'} \le \ul{w} \text{ with } v' \le m_{\ast}(\ul{w'})} \CZ_{v',  \ul{w'}, >0}$.
    
    (3) Let $\ul{v}  \in W^n $ with $m_\bullet(\ul{v}) = v$ and $\ul{v}$ be positive in $\ul{w}$. Let $\mathbf w_i$ be a reduced expression of $w_i$ and $(\mathbf v_i)_+$ be the positive subexpression for $v_i$ in $\mathbf w_i$ for all $i$. Then we have an isomorphism 
    \begin{align*}
    G_{(\mathbf v_1)_+, \mathbf w_1, >0} \times  G_{(\mathbf v_2)_+, \mathbf w_2, >0} \times \cdots \times G_{(\mathbf v_n)_+, \mathbf w_n, >0} &\to \CZ_{v,  \ul{w}, >0}, \\
    (g_1, g_2, \dots, g_n ) &\mapsto (g_1, g_2, \dots, g_n B^+/B^+).
    \end{align*}
\end{thm}

\begin{proof}
By Proposition~\ref{thm:closure}, we have  
\begin{align*}
 &\overline{U^-_{\ge 0} y_{\infty_1 }(\BR_{> 0})  U^-_{\ge 0} y_{\infty_2 }(\BR_{> 0})  U^-_{\ge 0} \dots y_{\infty_{n-1}}(\BR_{> 0})  U^-_{\ge 0}\tilde B^+/\tilde B^+} \\
 = &\bigcup_{\ul{w} \in W^n} \tilde{\CB}_{th(\ul{w}),\ge 0} = \!\!\!\!\!\bigsqcup_{\stackrel{\tilde{v} \le \tilde{w} \in \tilde{W}, }{\tilde{w} \le th({\ul{w}}) \text{ for some } \ul{w} \in W^n }}  \!\!\!\!\! \tilde{\CB}_{\tilde{v},\tilde{w},> 0}.
\end{align*}
 Therefore we have 
\begin{equation}\label{eq:a}
\overline{U^-_{\ge 0} y_{\infty_1 }(\BR_{> 0})  U^-_{\ge 0}   \dots y_{\infty_{n-1}}(\BR_{> 0})  U^-_{\ge 0}\tilde B^+/\tilde B^+} \cap \CZ' = \bigsqcup_{v \le m_\ast(\ul{w})}  \tilde{\CB}_{i(v), th(\ul{w}),> 0}.
\end{equation}

For any subset $S \subset \{1, \dots, l\}$, let $\theta_S: \tilde{G} \rightarrow \tilde{G}$ be the unique group automorphism of $\tilde{G}$ that is the identity in $G$, $\tilde{T}$, and maps $x_{\infty_l}(a)$ to $x_{\infty_l}((-1)^{\delta_{l \in S}}a)$ and $y_{\infty_l}(a)$ to $y_{\infty_l}((-1)^{\delta_{l \in S}}a)$ for all $1 \le l \le n-1$. Here $\delta_{l \in S} = 1$ if $ l \in S$ and $\delta_{l \in S} = 0$ if $ l \not \in S$.

The automorphism $\theta_S$ induces an automorphism in $\CZ'$, again denoted by $\theta_S$. Applying the automorphism $\theta_S$ to \eqref{eq:a}, we obtain 
\begin{align*}
&\overline{\pi^{-1}(U^-_{\ge 0}, U^-_{\ge 0}, \dots, U^-_{\ge 0}B^+/B^+)} \cap \CZ'  
 \\
 = &\,\,\overline{U^-_{\ge 0} y_{\infty_1 }(\BR^\ast)  U^-_{\ge 0}   \dots y_{\infty_{n-1}}(\BR^\ast)  U^-_{\ge 0}\tilde B^+/\tilde B^+} \cap \CZ' \\  
 =  &\bigsqcup_{v \le m_\ast(\ul{w})}  \Big( \bigsqcup_{S \subset \{1, \dots, l\}}   \theta_S \Big(\tilde{\CB}_{i(v), th(\ul{w}),> 0}\Big) \Big).
 \end{align*}
 
 The projection $\pi: \CZ' \rightarrow \CZ$ is invariant with respect to the automorphism $\theta_S$. We obtain that 
 \[
 	 \bigsqcup_{v \le v', \ul{w'} \le \ul{w} \text{ with } v' \le m_{\ast}(\ul{w'})} \CZ_{v',  \ul{w'}, >0} \subset \overline{\CZ_{v,  \ul{w}, >0}}.
 \]
The reverse inclusion follows by Proposition~\ref{thm:closure}. Part (2) is proved. 

Since each $ \theta_S  \Big(\tilde{\CB}_{i(v), th(\ul{w}),> 0}\Big)$ is a  connected component  in $\tilde{\CB}_{i(v), th(\ul{w})} (\BR)$, we see that  $\CZ_{v,  \ul{w}, >0}$ is a connected component of $\mathring{\CZ}_{v, \ul{w}}(\BR)$. This proves part (1).

We prove (3). Since $i (\ul{\mathbf v})$ is a positive subexpression of $th(\ul{\mathbf w})$, we have the isomorphism by \S\ref{Rp-BH} (a)
\[
	 G_{i (\ul{\mathbf v}), th(\ul{\mathbf w}), >0} \xrightarrow{\sim}   \mathring{\tilde \CB}_{i(v),th(\ul{w}), >0}.
\]
It follows from the definition  that 
\[
G_{i (\ul{\mathbf v}), th(\ul{\mathbf w}), >0}  \cong G_{(\mathbf v_1)_+, \mathbf w_1, >0} \times y_{\infty_1 }(\mathbb{R}_{>0}) \times  \cdots  \times y_{\infty_{n-1} }(\mathbb{R}_{>0})  \times G_{(\mathbf v_n)_+, \mathbf w_n, >0}
\]
 
By \cite[Lemma~5.1]{BH22}, we see that 
 \[
 \tilde{\CB}_{i(v), th(\ul{w}),> 0} \subset \dot{r}_1 U^- U^*_{-\alpha_{\infty_1}} \dot{r}_2 U^-  \cdots  \dot{r}_n U^- /  \tilde{B}^+, \text{for } \ul{r} \in W^n \text{ with } \ul{r}  \le \ul{w},  v \le m_\bullet(\ul{r}).
 \]
 The section in \eqref{eq:th} restricts to an isomorphism 
 \[
 G_{(\mathbf v_1)_+, \mathbf w_1, >0} \times y_{\infty_1 }(\mathbb{R}_{>0}) \times  \cdots  \times y_{\infty_{n-1} }(\mathbb{R}_{>0})  \times G_{(\mathbf v_n)_+, \mathbf w_n, >0} \rightarrow \CZ_{v,\ul{w}, >0} \times \BR^{n-1}_{>0}.
 \]
 Note that we have natural $\tilde{T}$-action on both sides and the isomorphism is $\tilde{T}$-equivariant. Applying the argument in \cite[Proposition~4.2]{BH22} $n-1$ times, we obtain the desired isomorphism in (3). 
\end{proof}

\subsection{Stratified isomorphisms}
 

By \cite[Lemma~5.1 \& Theorem~5.3]{BH22}, we have the following commutative diagram for any $u \in \tilde{W}$
\begin{equation}\label{eq:prod}
\begin{tikzcd}[column sep=small]
\tilde{\CB}_{\ge0} \cap \dot{u} \tilde{U}^- \tilde{B}^+  \ar[d, "\cong"] \ar[r, equal]&  \bigsqcup_{v \le u \le w} \tilde{\CB}_{v,w, >0}  \ar[d, "\cong"]   \ar[r, hookleftarrow]&   \tilde{\CB}_{v,w, >0} \ar[d, "\cong"] \\
\mathring{\tilde{\CB}}_{u, \ge0} \times  \mathring{\tilde{\CB}}^{u}_{\ge0} \ar[r, equal] &  \bigsqcup_{v \le u } \tilde{\CB}_{v,u, >0} \times  \bigsqcup_{ u \le w} \tilde{\CB}_{u, w >0}  \ar[r, hookleftarrow]&\tilde{\CB}_{v,u, >0} \times    \tilde{\CB}_{u, w >0}.
\end{tikzcd}
\end{equation}

Here $\mathring{\tilde{\CB}}_{u, \ge0} = \mathring{\tilde{\CB}}_{u} \cap \tilde{\CB}_{\ge 0}$ and $\mathring{\tilde{\CB}}^u_{\ge0} = \mathring{\tilde{\CB}}^{u} \cap \tilde{\CB}_{\ge 0}$. We have the following similar result for the twisted product.  

\begin{prop}\label{prop:TM}
Let $(u, \ul{x}) \in Q_n$ be a minimal element. We have the following commutative diagram
\[ 
\begin{tikzcd}
\displaystyle \!\!\!\!\!\!\!\!  \!\!\!\!\!\!\!\!  \bigsqcup_{ v' \le u, \ul{x} \le \ul{w'} \text{ with } v' \le m_{\ast}(\ul{w'})}  \!\!\!\!\!\!\!\!  \!\!\!\!\!\!\!\! \mathring{\CZ}_{v',\ul{w'}, >0}  \ar[r, "\cong"] \ar[d, hookleftarrow]&\displaystyle \bigsqcup_{v' \le u}  \tilde{\CB}_{i(v'), i(u),>0}  \times  \bigsqcup_{\ul{x} \le \ul{w'}} \tilde{\CB}_{ th(\ul{x}),th(\ul{w'}), >0} 
   \ar[d, hookleftarrow]\\
  \mathring{\CZ}_{v',\ul{w'}, >0}  \ar[r, "\cong"] &   \tilde{\CB}_{i(v'), i(u),>0}  \times   \tilde{\CB}_{ th(\ul{x}),th(\ul{w'}), >0}.
\end{tikzcd}
\]
 
In particular, $\CZ_{v,\ul{w}, \ge 0} $ is a topological manifold with boundary 
\[
 \partial\CZ_{v,\ul{w}, \ge 0} =\bigsqcup_{(v', \ul{w'}) <(v, \ul{w} ) \text{ in } Q_n} \mathring{\CZ}_{v',\ul{w'}, >0} . 
\]

\end{prop}

\begin{remark}
By  Theorem~\ref{thm:Zp},  the left hand side is open in $\CZ_{\ge 0}$. By Proposition~\ref{thm:closure}, the right-hand side is locally closed in $\tilde{\CB} \times \tilde{\CB}$.
\end{remark}
 
 \begin{proof}
 Let $\ul{x} = (x_1, \dots, x_n)$ and $\tilde{x} = th(\ul{x}) \in \tilde{W}$. Note that 
\begin{align*}
\tilde{\CB}_{\ge0} \cap \dot{\tilde{x} } \tilde{U}^- \tilde{B}^+ \cap \CZ' = \bigsqcup_{v \le m_\ast(\ul{x}),  \ul{x} \le \ul{w}} \tilde{\CB}_{i(v),th(\ul{w}), >0} \\
\subset \dot{x}_1 U^- U^*_{-\alpha_{\infty_1}} \dot{x}_2 U^- U^*_{-\alpha_{\infty_2}} \cdots \dot{x}_n U^- \tilde{B}^+/ \tilde{B}^+.
\end{align*}

By \eqref{eq:th}, we have a  section
 \begin{align*} 
	  \,\, \prod_{i=1}^n U_{-\a_{\infty_i }}^* (\mathbb{R}_{>0})    \times  \!\!\!\!    \bigsqcup_{ v' \le u, \ul{x} \le \ul{w}}  \!\!\!\!  \mathring{\CZ}_{v',\ul{w'}, >0}  \rightarrow   \bigsqcup_{v \le m_\ast(\ul{x}),  \ul{x} \le \ul{w}} \tilde{\CB}_{i(v),th(\ul{w}), >0}. 
 \end{align*}
 By \eqref{eq:prod}, we further have the embedding
 \[
\prod_{i=1}^n U_{-\a_{\infty_i }}^* (\mathbb{R}_{>0})   \times  \!\!\!\!    \bigsqcup_{ v' \le u, \ul{x} \le \ul{w}}  \!\!\!\!\mathring{\CZ}_{v',\ul{w'}, >0}  \rightarrow  \mathring{\tilde{\CB}}_{i(u), \ge0} \times  \mathring{\tilde{\CB}}^{i(u)}_{\ge0}.
\]

Applying  \eqref{eq:prod} again, we obtain that 
\begin{align}\label{eq:1}
&\displaystyle \prod_{i=1}^n U_{-\a_{\infty_i }}^* (\mathbb{R}_{>0})   \times \!\!\!\!    \bigsqcup_{ v' \le u, \ul{x} \le \ul{w}}  \!\!\!\! \mathring{\CZ}_{v',\ul{w'}, >0}  
\rightarrow 
 \mathring{\tilde{\CB}}_{i(u), \ge0} \times ( \mathring{\tilde{\CB}}^{i(u)}_{\ge0} \cap \dot{\tilde{x}} \tilde{U}^- \tilde{B}^+) \\
\notag &\rightarrow   \mathring{\tilde{\CB}}_{i(u), \ge0} \times  \tilde{\CB}_{i(u),th(\ul{x}), >0} \times   \mathring{\tilde{\CB}}^{th(\ul{x})}_{\ge0} \cong  \mathring{\tilde{\CB}}_{i(u), \ge0} \times  \BR_{> 0}^{n-1} \times   \mathring{\tilde{\CB}}^{th(\ul{x})}_{\ge0}.
\end{align} 
Note that $\tilde{\CB}_{i(u),th(\ul{x}), >0} \cong \BR^{n-1}_{>0}$ since $(u, \ul{x}) \in Q_n$ is minimal. The composition in \eqref{eq:1} is a  $\tilde{T}$-equivariant isomorphism.   The desired stratified isomorphism follows by applying the argument in \cite[Proposition~4.2]{BH22} repeatedly. 

It follows from \cite[\S5.1]{BH22} that $\CZ_{v,\ul{w}, \ge 0} $ is a topological manifold with boundary 
\[
 \partial\CZ_{v,\ul{w}, \ge 0} =\bigsqcup_{(v', \ul{w'}) <(v, \ul{w} ) \text{ in } Q_n} \mathring{\CZ}_{v',\ul{w'}, >0}. \qedhere
\]
\end{proof}

\subsection{The duality map}
In this subsection, we assume that $G$ is a connected reductive group. We still write $\CZ = \CZ_n$.
 
Let $\iota$ be the involutive group automorphism on $G$ which fixes $T$ and sends $x_i(a)$ to $x_i(-a)$, $y_i(a)$ to $y_i(-a)$ for $i \in I$ and $a \in \BC$. Then $\iota$ preserves the Borel subgroup $B^+$ and induces an involution on $\CB$. 

We define $$\Phi: \CB \to \CB, \qquad g B^+/ B^+ \mapsto \iota(\dot w^{-1}_0 g)  B^+/ B^+.$$

It is easy to see that $\Phi$ is an involution on $\CB$ and for any $v \le w$ in $W$, we have $\iota(\OCB_{v, w})=\OCB_{v, w}$ and $\Phi(\mathring{\CB}_{v, w})=\mathring{\CB}_{w_0 w, w_0 v}$. Since $\iota(\dot w^{-1}_0 U^-_{>0}) B^+/ B^+= U^+_{>0} \dot{w}_0 B^+/ B^+ = U^-_{>0} B^+/ B^+$ by \cite[Theorem~8.7]{Lus-1}, $\Phi$ restricts to an isomorphism on $\CB_{\ge 0}$.

 We define the involutions 
\begin{gather*} \Phi: \CB^n \to \CB^n, \quad (h_1  B^+/ B^+, \dots, h_n  B^+/ B^+ ) \mapsto (\Phi(h_n  B^+/ B^+), \dots,  \Phi(h_1  B^+/ B^+)), \\
\Phi: \CZ  \to \CZ , \quad (g_1, g_2, \dots, g_n) \mapsto (\iota(\dot w^{-1}_0 g_1 g_2\cdots g_n), \iota(g_n \i), \dots, \iota(g_2 \i)).
\end{gather*}

 Recall the isomorphism $\alpha: \CZ \rightarrow \CB^n$. Then we have the commutative diagram 
\[
\xymatrix{
\CZ  \ar[r]^-{\Phi} \ar[d]^-{\a} & \CZ  \ar[d]^-{\a} \\
\CX \ar[r]^-{\Phi} & \CX.
}
\]

\begin{prop}\label{prop:duality}
    We have $\Phi(\CZ_{  \ge 0})=\CZ_{  \ge 0}$. 
\end{prop}

\begin{proof}

Note that $U^-_{>0}$ is Hausdorff dense in $U^-_{\ge 0}$. Thus, $\CZ_{\ge 0}$ is the Hausdorff closure of $(U^-_{>0}, \dots,  U^-_{>0})$ in $\CZ$. By definition, $\iota(U^-_{>0}) \i=U^-_{>0}$. 

Let $g_1, g_2, \dots, g_n \in U^-_{>0}$. By \cite[Lemma~5.6]{GKL}, we have $\dot w^{-1}_0 g_1=g'_1 b_1 \text{ for some } g'_1 \in U^-, b_1 \in B^+_{>0}$. We write $b_1 g_2 = g_2' b_2$ for some $g'_2 \in U^-_{>0}$, $b_2 \in B^+_{>0}$. We similarly obtain $g'_i \in U^-_{>0}$, $b_i \in B^+_{>0}$ for all $2 \le i \le n$.

Then $\iota( g_1'  \cdots g'_n) B^+ = \iota(\dot w^{-1}_0 g_1\cdots g_n) B^+ = \iota(\dot w^{-1}_0 g_1g_2\dot w_0 ) \dot w_0 B^+ \subset U^+_{>0} \dot{w}_0 B^+ $. Hence $\iota( g_1'  \cdots g'_n) B^+ \in U^-_{>0} B^+$ by \cite[Lemma~5.6]{GKL}. It follows that $\iota( g_1'  \cdots g'_n) \in U^-_{>0}$ since $\iota( g_1'  \cdots g'_n) \in U^-$. 
 Therefore we have 
\begin{align*}
 &\, (\iota(\dot w^{-1}_0 g_1 g_2\cdots g_n), \iota(g_n \i), \dots, \iota(g_2 \i)) \\
 = & \, (\iota(g'_1 \cdots,  g'_n), \iota(g'_n)\i, \dots, \iota(g'_2)\i   ) \in (U^-_{>0}, U_{>0}^-, \dots, U^-_{>0}) \subset \CZ_{\ge 0}.
\end{align*}
Hence $\Phi(\CZ_{\ge 0}) \subset (\CZ_{\ge 0})$. Since $\Phi$ is an isomorphism, we must have $\Phi(\CZ_{\ge 0})= (\CZ_{\ge 0})$. This finishes the proof.
\end{proof}

For any $v \in W$ and $\ul{w} = (w_1, w_2, \dots, w_n) \in W^n$ with $v \le m_\ast(\ul{w})$, it follows by direct computation and the above proposition that 
\begin{equation*}
\Phi(\CZ_{v, (w_1, w_2,  \dots, w_n), > 0})= {\CZ}_{w_0 w_1,(w_0 v, w_n\i, \dots, w_2 \i), >0}.
\end{equation*}


\section{Regularity theorem}

\subsection{Some properties on  posets}\label{subsec:posets}
 Let $(P, \le)$ be a poset. For any $x, y \in P$ with $x \le y$, the interval between $x$ and $y$ is defined as $[x, y]=\{z \in P \mid x \le z \le y\}$. A subset $C \subset P$ is called {\it convex} if for any $x, y \in C$, we have $[x,y] \subset C$.

The covering relation is denoted by $\gtrdot$. For any $x, y \in P$ with $x \le y$, a maximal chain from $x$ to $y$ is a finite sequence of elements $y=w_0 \gtrdot w_1 \gtrdot \cdots \gtrdot w_n=x$. In general, the maximal chain may not exist. 

The $P$ is called {\it pure} if for any $x, y \in P$ with $x \le y$, the maximal chains from $x$ to $y$ always exist and have the same length. A pure poset $P$ is called {\it thin} if every interval of length $2$ has exactly $4$ elements, i.e., it has exactly two elements between $x$ and $y$. 

The order complex $\Delta_{ord}(P)$ of a finite poset $P$ is a simplicial complex whose vertices are the elements of $P$ and whose simplices are the chains $x_0 < x_1 < \cdots <x_k$ in $P$.  A finite pure poset is called {\it shellable} if its order complex $\Delta_{ord}(P)$  is shellable, that is, its maximal faces can be ordered as $F_1, F_2, \dots, F_n$ so that $F_k \cap \big( \cup_{i=1}^{k-1} F_i\big)$ is a nonempty union of the maximal proper faces of $F_k$ for $k = 2, \dots, n$. 


\subsection{Regular CW complexes}\label{subsec:CW}
Let $X$ be a Hausdorff space. We call a finite disjoint union $X = \bigsqcup_{\alpha \in P}X_{\alpha}$ a {\it regular CW complex} if it satisfies the following two properties.

\begin{enumerate}
\item For each $\alpha \in P$, there exists a homeomorphism from a closed ball to $\overline{X}_\alpha$ mapping the interior of the ball to $X_\alpha$. 
\item For each $\alpha$, there exists $P' \subset P$, such that $\overline{X}_\alpha = \bigsqcup_{\beta \in P'} X_\beta$.
\end{enumerate}
We then have a natural poset $({P}, \le)$, where $\beta \le \alpha$ if and only if $X_\beta \subset \overline{X}_\alpha$. The face poset of $X$ is defined to be the augmented poset $\hat{P} = P \sqcup \{\hat{0}\}$, where $\hat{0}$ is the added unique minimal element of $\hat{P}$.

One of the motivations for the combinatorial properties in \S\ref{subsec:posets} comes from the following result in \cite[Proposition~2.2]{Bj2}. 

(a) {\it Let $\hat{0} \neq x \in \hat{P} $. If the closed interval $[\hat{0},x]$ is pure, thin, and shellable, then it is the face poset of some regular CW complex which is homeomorphic to a closed ball.}


We also recall the following result from \cite{Bj2} that will be used in Proposition~\ref{thm:CW}.

(b) {\it 
Suppose that $X$ is a regular CW complex with (finite) face poset $\hat{P}$. If $\hat{P}\sqcup \{ \hat{1}\}$ (adjoining  a maximum $\hat{1}$) is graded, thin, and shellable, then $X$ is homeomorphic to a sphere of dimension rank$(P)-1$.}

\subsection{Index set of the stratification of $\CZ_n$}

Let 
$$
Q_n =\{(v, \ul{w}) \in W \times W^n; v \le m_{\ast}(\ul{w})\}.
$$ The partial order on $Q_n$ is defined by $(v', \ul{w'}) \le (v, \ul{w})$ if $v \le v'$ and $\ul{w'} \le \ul{w}$. The augmented poset $\hat{Q}_n = Q_n \sqcup \{\hat{0}\}$ is the face poset of $\CZ$.

\begin{prop}
\label{thm:hatQ}
    The augmented poset $\hat{Q}_n$ is pure and thin. 
\end{prop}

\begin{proof}
The purity is clear. We write $\hat{Q} = \hat{Q}_n$ and $Q = Q_n$ in the proof. We define  
$$
\tilde Q=\{(x, y) \in \tilde W^{op} \times \tilde W; x \le y\}.
$$
 The partial order on $\tilde Q$ is defined by $(x', y') \le (x, y) $ if $x \le x' \le y' \le y$. Define the map 
$$
f: Q  \to \tilde Q, \qquad (v, \ul{w}) \mapsto (i(v), th(\ul{w})).
$$ 

The map $f$ identifies $Q $ with a sub-poset of $\tilde Q$.
Note that the image of $f$ is $$\text{Im}(f)=\{(x, y) \in Q; x \in W, y \in W s_{\infty_1 } W  s_{\infty_2} W \cdots s_{\infty_{n-1}} W\}.$$ It is easy to see that $\text{Im}(f)$ is a convex subset of $\tilde Q$. 

Next, we apply \cite{Dyer} to the Coxeter group $\tilde{W} \times \tilde{W}$ with respect to the reflection order $\le_A$, where $A$ is the set of reflections in $\tilde{W} \times \{e\}$.  It follows from the definition in \cite[\S1]{Dyer} that $(x',y') \le_A (x,y)$ if and only if $x\le x'$ and $y'\le y$. Therefore, $\tilde{Q}$ is a convex subset of $(\tilde{W} \times \tilde{W}, \le_A)$. By \cite[Proposition~2.5]{Dyer}, $\tilde Q$ is thin. It remains to consider the intervals of length $2$ with the minimal element $\hat{0}$. Suppose that the maximal element in this interval is $(v, \ul{w})$. Then $\ell(\ul{w})-\ell(v)=1$. 

Let $\ul{w}=(w_1, w_2, \dots, w_n)$ and $\mathbf w_i$ be a reduced expression of $w_i$ for each $i$. We write $(\mathbf w_1,   \mathbf w_2, \dots, \mathbf w_n)$ as $(s_1, \ldots, s_m)$. Then there exists $1 \le l \le m$ such that $v=s_1 \cdots \hat s_l \cdots s_m$. We show that 

(a) {\it $\sharp\{l; v=s_1 \cdots \hat s_l \cdots s_m\}=1$ or $2$. }

Suppose that there exists $l_1<l_2$ such that $v=s_1 \cdots \hat s_{l_1} \cdots s_m=s_1 \cdots \hat s_{l_2} \cdots s_m$.  Then we have $s_{l_1} \cdots s_{l_2-1}=s_{l_1+1} \cdots s_{l_2}$. Thus $$s_1 \cdots s_m=s_1 \cdots \hat s_{l_1} \cdots \hat s_{l_2} \cdots s_m<v.$$ Also we have
\begin{align*} v=m_{\ast}(\ul{w}) &=s_1 \ast \cdots \ast s_{l_1} \ast \cdots \ast s_{l_2} \ast \cdots \ast s_m \\ &=s_1 \ast \cdots \ast \hat s_{l_1} \ast \cdots \ast s_{l_2} \ast \cdots \ast s_m.\end{align*}

Suppose that there exists $l_1<l_2<l_3$ such that $v=s_1 \cdots \hat s_{l_j} \cdots s_m$ for $j=1, 2, 3$. Then we have 
\begin{align*} v=m_{\ast}(\ul{w}) &=s_1 \ast \cdots \ast s_{l_1} \ast \cdots \ast s_{l_2} \ast \cdots \ast s_{l_3} \ast \cdots \ast s_m \\ &=s_1 \ast \cdots \ast \hat s_{l_1} \ast \cdots \ast s_{l_2} \ast \cdots \ast s_{l_3} \ast \cdots \ast s_m \\ 
&=s_1 \ast \cdots \ast \hat s_{l_1} \ast \cdots \ast \hat s_{l_2} \ast \cdots \ast s_{l_3} \ast \cdots \ast s_m.
\end{align*}

In particular, $\ell(v) \le m-2$. That is a contradiction. (a) is proved. We now analyze the two cases separately.

Case (1): $\{l; v=s_1 \cdots \hat s_l \cdots s_m\}=\{l_1\}$. 

The expression $(s_1, \ldots, \hat s_{l_1}, \ldots, s_m)$ defines an element $\ul{w_1} \in W^n $ with $\ul{w_1}<\ul{w}$. It is also easy to see that $s_1 \ast \cdots \ast s_m>v=m_\ast(\ul{w_1})=s_1 \ast \cdots \ast \hat s_{l_1} \ast \cdots \ast s_m$. In this case, there are exactly two elements between $\hat{0}$ and $(v, \ul{w})$. They are $(v, \ul{w_1})$ and $(m_{\ast}(\ul{w}), \ul{w})$. 
 
Case (2): $\{l; v=s_1 \cdots \hat s_l \cdots s_m\}=\{l_1, l_2\}$. 

The expressions $(s_1, \ldots, \hat s_{l_1}, \ldots, s_m)$ and $(s_1, \ldots, \hat s_{l_2}, \ldots, s_m)$ define two elements $\ul{w_1}$, $\ul{w_2} \in W^n$ with $\ul{w_1}, \ul{w_2}<\ul{w}$. In this case, there are exactly two elements between $\hat 0$ and $(v, \ul{w})$. They are $(v, \ul{w_1})$ and $(v, \ul{w_2})$. 
\end{proof}

\subsection{Regularity theorem}\label{subsec:regular}

 

 Recall that an $n$-dimension topological manifold with boundary is a Hausdorff space $X$ such that every point $x \in X$ has an open neighborhood homeomorphic to either $\BR^n$, or $\BR_{\ge 0} \times \BR^{n-1}$ mapping $x$ to a point in $\{0\} \times  \BR^{n-1}$. In the latter case, we say that $x$ is on $\partial X$, the boundary of $X$.

The following statement can be derived from the generalized Poincar\'e conjecture and Brown's collar theorem. We refer to \cite[\S 3.2]{GKL} for details.

\begin{thm}\label{thm:poincare}  Let $X$ be a compact $n$-dimensional topological manifold with boundary, such that $\partial X$ is homeomorphic to an $(n-1)$-dimensional sphere and $X - \partial X$ is homeomorphic to an $n$-dimensional open ball. Then $X$ is homeomorphic to an $n$-dimensional closed ball. 
\end{thm}

We have the following regularity result.
\begin{prop}\label{thm:CW}
Let $v \in W$ and $\ul{w} \in W^n$ with $v \le m_{\ast}(\ul{w})$.  Suppose that the interval $[\hat{0}, (v,\ul{w})]$ in $Q_n$ is shellable. Then $\CZ_{n,v,\ul{w}, \ge 0}$ is a regular CW complex homeomorphic to a closed ball.
\end{prop}
\begin{proof}
We write $\CZ= \CZ_n$.
We proceed by induction on $\ell(\ul{w}) - \ell(v)$. The statement is clear when $\ell(\ul{w}) - \ell(v) =0$, since $\CZ_{v,\ul{w}, \ge 0}$ is a single point in this case.

Assume $\partial\CZ_{v,\ul{w}, \ge 0} $ is a regular CW complex with face poset 
\[
	 Q'_n = \{\alpha \in \hat{Q}_n \vert  \alpha <(v, \ul{w} ) \}.
\]
Adding a maximal element $\hat{1}$, we have $ Q_n' \sqcup \{\hat{1} \} \cong  [\hat{0}, (v, \ul{w} ) ] \subset \hat{Q}_n$ as posets. By \S\ref{subsec:CW} (b) and Proposition~\ref{thm:hatQ}, $\partial\CZ_{v,\ul{w}, \ge 0}$ is a regular CW complex homeomorphic to a sphere of dimension $\ell(\ul{w}) - \ell({v})$.  Now the result follows from Theorem~\ref{thm:poincare}.
\end{proof}

\subsection{EL-shellability of $\hat Q_2$}\label{subsec:ELQ2}
Let $n=2$ and $Q_2 = \{(v, \ul{w}) \in W \times W^2; v \le m_{\ast}(\ul{w})\}$ in this section. 
This subsection is devoted to proving the following EL-shellability of $\hat Q_2$. 

\begin{thm}\label{thm:conj}
For any $x \in \hat{Q}_2$, the interval $[\hat{0},x]$ is EL-shellable.
\end{thm}
This subsection is organized as follows: we first recall EL-shellability in \S\ref{subsec:ELpre}; we then show in \S\ref{subsec:LW} that the augmented poset $\tilde{Q}_2^\wedge$ (defined below) is EL-shellable following \cite{LW}  with some flexibility on the labelling;	 we then consider in \S\ref{subsec:T} a more specific labeling and prove some of its basic properties;  finally, in \S\ref{subsec:EL} we label the edges of $\hat{Q}_2$ as \eqref{eq:EL} and prove EL-shellability.

 \subsubsection{}\label{subsec:ELpre}
 
Suppose that $P$ is a pure poset. An edge labeling of $P$ is a map $\l$ from the set of all covering relations in $P$ to a poset $\L$. The labeling $\l$ sends any maximal chain $(y \gtrdot y_1  \gtrdot \cdots  \gtrdot y_n=x)$ of an interval $[x,y]$ of $P$ to a tuple $(\lambda(y, y_1),\lambda(y_1 , y_2), \dots, \lambda(y_{n-1},  y_n) )$ of $\L$. A maximal chain is called {\it increasing} if the associated tuple of $\L$ is increasing. The edge labeling $\l$ also allows one to order the maximal chains of any interval of $P$ by ordering the corresponding tuples lexicographically. 

An edge labeling of $P$ is called an EL-labeling {\it}\footnote{The notion of EL-labeling used in this paper is consistent with \cite{LW} and is dual to the original label introduced in \cite{Bj}.} (edge lexicographical labeling) if for every closed interval there exists a unique increasing maximal chain, and all other maximal chains of this interval are less than this maximal chain (with respect to the lexicographical order). If a pure poset admits an EL-labeling, it is automatically shellable by \cite[Theorem~3.3]{BW}. We say that such a poset is {\em EL-shellable}.
 
The following proposition will be used later to check the EL-labeling later.  

\begin{prop}\cite[Proposition~2.5]{Bj}\label{prop:EL}
Let $\lambda$ be an edge labeling on a graded poset $P$ such that every interval $[x,y]$ contains a unique increasing maximal chain, denoted by $c(x,y) = (y \gtrdot y_1  \gtrdot \cdots  \gtrdot y_n =x)$.

Then such labeling is an EL-labeling if and only if for any such $c(x,y)$ and $y \gtrdot y' \ge x$, we have $\lambda(y, y_1) <  \lambda(y, y')$.
\end{prop}

\subsubsection{}\label{subsec:LW}

 Let $\tilde{Q}_2 =\{(x, y) \in \tilde{W}^{op} \times \tilde{W}; x \le y\}$. We define the augmented poset $\tilde{Q}_2^\wedge=\tilde{Q}_2  \sqcup \{\tilde{0}^\wedge\}$, where $\tilde{0}^\wedge$ is the added minimal element. 
The EL-shellability of the poset $\tilde{Q}_2^\wedge$ is due to Williams \cite{LW}. We sketched her construction here, as her ideas will also be used in our situation. 

The edges of $\tilde{Q}_2^\wedge$ are of the following three types:
\begin{itemize}
    \item Type I: $(x, y) \gtrdot (x, y')$, where $y \gtrdot y'$ in $\tilde{W}$;
    \item Type II: $(x, y) \gtrdot (x', y)$, where $x \gtrdot x'$ in $\tilde{W}^{op}$;
    \item Type III: $(x, x) \gtrdot \tilde{0}^\wedge$. 
\end{itemize}

By \cite{Dyer}, both $\tilde{W}^{op}$ and $\tilde{W}$ are EL-shellable. In other words, there exists EL-labelings $\l_1$ (resp. $\l_2$) from the covering relations in $\tilde{W}$ (resp. $\tilde{W}^{op}$) to a poset $\L_1$ (resp. $\L_2$).

Let $\L=\L_1 \sqcup \{\emptyset\} \sqcup \L_2$ be the poset with partial order $\le$, where the restriction of $\le$ to $\L_i$ coincide with the given partial order on $\L_i$ and $a<\emptyset<b$ for any $a \in \L_1$ and $b \in \L_2$. 
The edge labelling on $\tilde{Q}_2^\wedge$, denoted by $\tilde{\lambda}$, is defined as follows: 
\begin{itemize}
\item the type I edge $(x, y) \gtrdot (x, y')$ is labelled by $\l_1(y \gtrdot y')$; 
\item the type II edge $(x, y) \gtrdot (x', y)$ is labelled by $\l_2(x \gtrdot x')$;
\item  the type III edges are labelled by $\emptyset$. 
\end{itemize}

It is easy to see that the restriction of this labeling to the poset $\tilde{Q}_2$ is an EL-labeling. It remains to consider the interval $[\tilde{0}^\wedge, (x, y)]$ for any $(x, y) \in \tilde{Q}_2$.

Note that $\emptyset<b$, for any $ b\in\L_2$. For any $(x, y) \in \tilde{Q}_2$, an increasing maximal chain in $[\tilde{0}^\wedge, (x, y)]$ must be of the form  $(x,y) \gtrdot (x,y_1)\gtrdot \cdots \gtrdot (x,x) \gtrdot \tilde{0}^\wedge$, where $y \gtrdot y_1 \gtrdot \cdots \gtrdot y_n=x$ is a maximal chain in $W$ with increasing labelings. The existence and uniqueness of such a chain $y \gtrdot y_1 \gtrdot \cdots \gtrdot y_n=x$ is due to the EL-labeling of $\tilde{W}$. In particular, there is a unique increasing maximal chain in $[\tilde{0}^\wedge, (x, y)]$. 

Recall $a < b $ for any $a \in \L_1$ and $b \in \L_2$. We have $\tilde{\lambda}\Big((x, y_i) \gtrdot (x, y_{i+1})\Big)<\tilde{\lambda}\Big((x, y_i) \gtrdot (x', y_i)\Big)$ for any $x \gtrdot x'$ in $W^{op}$. Moreover, by the EL-labeling of $W$, $\tilde{\lambda}\Big((x, y_i) \gtrdot (x, y_{i+1})\Big)<\tilde{\lambda}\Big((x, y_i) \gtrdot (x, y')\Big)$ for any $y_i \gtrdot y'$ in $W$ with $y' \neq y_{i+1}$. Now by \cite[Proposition~2.5]{Bj}, the labeling $\tilde{\lambda}$ on $\tilde{Q}_2^\wedge$ is an EL-labeling. 

For any $\alpha \le \beta $ in $\tilde{Q}_2^\wedge$, we denote by $c(\beta, \alpha)$ the unique increasing maximal chain from $\alpha$ to $\beta$.

\subsubsection{}\label{subsec:T}
We consider a specific EL-labeling of $\tilde{W}$. 

Let $ \L_1=\tilde{T}$ be the set of reflections in $\tilde{W}$. By \cite[(2.3)]{Dyer1}, there exists an reflection order on $\tilde{T}$ such that 
\[
	t' \prec t, \quad \text{for any } t \in \tilde{T} - W_I \text{ and }  t' \in \tilde{T}\cap W_I.
\]

 We define the edge labeling $\lambda_1$ of $\tilde{W}$ by $\lambda_1 (y \gtrdot y') = y y'^{-1}$. It follows from \cite[Proposition~3.9]{Dyer} that $\lambda_1$  defines an EL-labeling of $\tilde{W}$.

 We show that

	 (a) {\it for $y \gtrdot y_1$ and $y \gtrdot y_2$ with $y_1 \in th(W)$ and $y_2 \not \in th(W)$, we have $yy_{1}^{-1} \prec  yy_{2}^{-1}$ in $\tilde{T}$.}
	 
	 There are two cases to consider. We write $s_\infty = s_{\infty_1} \in \tilde{W}$.
	 
	 Case I: $y = as_{l}b s_{\infty} c$, $y_1 =  ab s_{\infty} c$, $y_2 =  as_{l} b c$ for some $a,b,c \in W_I$. 
 We write $t_1 = y y_1^{-1} = as_la^{-1}\in \tilde{T}$, $t_2 = yy_2^{-1} = as_{l}b s_{\infty} (as_{l}b)^{-1} \in \tilde{T}$. It follows by \cite[\S3]{Dyer3}, the subgroup of $\tilde{W}$ generated by $t_1, t_2$ is a dihedral group, with Coxeter generators $t_1, abs_{\infty}(ab)^{-1}$. Since $t_1\in W_I$, we have $t_1 \prec abs_{\infty} (ab)^{-1} $. By definition of reflection orders \cite[\S2.1]{Dyer1}, we have  $t_1 \prec t_2 = t_1 abs_{\infty} (ab)^{-1} t_1^{-1}$.
 

Case II: $y = as_{\infty} b s_{l} c$, $y_1 =  as_{\infty} b  c$, $y_2 =  ab s_{l}  c$ for some $a,b,c \in W_I$.
 We write $t_1 = y y_1^{-1} = as_{\infty} b s_l (as_{\infty} b)^{-1}$, $t_2 = yy_2^{-1} = as_{\infty} a^{-1} \in \tilde{T}$. It follows by \cite[\S3]{Dyer3}, the subgroup of $\tilde{W}$ generated by $t_1, t_2$ is a dihedral group, with Coxeter generators $abs_l (ab)^{-1}, t_2$. Since $abs_l (ab)^{-1}  \prec t_2$, we have $t_2 abs_l (ab)^{-1} t_2^{-1}=t_1 \prec t_2$ by definition of reflection orders \cite[\S2.1]{Dyer1}.
 
 
Now (a) is proved.
 
\subsubsection{}\label{subsec:EL}
Now we construct an EL-labeling on the poset $\hat{Q}_2$.  Define the map  
\[
f: \hat{Q}_2  \longrightarrow \tilde{Q}_2^\wedge, \quad (v, \ul{w}) \mapsto (i(v), th(\ul{w})),\quad  \hat{0} \mapsto \tilde{0}^\wedge.
\]

The map $f$ identifies $\hat{Q}_2 $ with a sub-poset of $\tilde{Q}_2^\wedge$.  However, for $(v, \ul{w}) \gtrdot \hat{0} $ in $\hat{Q}_2$, we do not have $(i(v), th(\ul{w})) \gtrdot \tilde{0}^\wedge $ in $\tilde{Q}_2^\wedge$. So results in \cite{LW} cannot apply directly.  

We consider the labeling on $\tilde{Q}_2^\wedge$ in \S\ref{subsec:LW}, where the edge labeling on $\tilde{W}$ is the specific one defined as in \S\ref{subsec:T}. There is no additional restriction on the EL-labeling on $\tilde{W}^{op}$. We now define an edge labelling $\lambda: \hat Q_2  \rightarrow \Lambda$ by 
\begin{gather}\label{eq:EL}
\lambda\Big((v, \ul{w}) \gtrdot (v', \ul{w}') \Big) = \tilde{\lambda} \Big((i(v), th(\ul{w})) \gtrdot (i(v'), th(\ul{w}'))  \Big), \\
\notag \lambda \Big((v, \ul{w}) \gtrdot \hat{0}\Big) = 
\text{the first edge label in the chain }c \Big((i(v), th(\ul{w})), \tilde{0}^\wedge \Big) \subset \tilde{Q}_2^\wedge \text{ via }\tilde{\lambda}.
\end{gather}

We claim this to be an EL-labeling of $\hat{Q}_2$. Thanks to the embedding of posets $Q_2 \rightarrow \tilde{Q}_2$ and \S\ref{subsec:LW}, the interval $[(v', \ul{w}'), (v, \ul{w})]$ contains a unique increasing maximal chain, which is lexicographically maximal.
It remains to consider the maximal chains in $[\hat{0}, (v,\ul{w})] \subset \hat{Q}_2$. 

	(a) {\it There exists a unique increasing maximal chain in $[\hat{0}, (v,\ul{w})]$.}
	
	We first show that the intersection of the maximal chain $c\Big((i(v), th(\ul{w})), \tilde{0}^\wedge \Big) \subset \tilde{Q}_2^\wedge$, constructed in \S\ref{subsec:LW}, with $f(\hat{Q}_2)$ is actually a maximal chain in $f(\hat{Q}_2)$. 
	
	Recall \S\ref{subsec:LW} that 
	\[
	c\Big((i(v), th(\ul{w})), \tilde{0}^\wedge \Big) = \Big((i(v), th(\ul{w}))   \gtrdot (i(v), \tilde{w}_1) \gtrdot  \cdots  \gtrdot  (i(v), i(v))\gtrdot \tilde{0}^\wedge\Big).
	\]
	 It follows from \S\ref{subsec:T} (a) and Proposition~\ref{prop:EL} that  $c\Big((i(v), th(\ul{w})), (i(v), i(v)) \Big) \cap f({Q_2})$ contains a minimal element in $f(Q_2)$, denoted by $ (i(v), th(\ul{w_n}))$.  
	 Then by the convexity of $f({Q_2})$ in $\tilde{Q}_2$, we have 
	 \begin{align*}
	 &c\Big((i(v), th(\ul{w})), (i(v), i(v)) \Big) \\
	 = & (i(v), th(\ul{w})) \gtrdot  (i(v), th(\ul{w_1})) \gtrdot \cdots \gtrdot  (i(v), th(\ul{w_n})) \gtrdot \cdots \gtrdot  (i(v), i(v)).
	 \end{align*}
	Then it is clear that $(v, \ul{w}) \gtrdot  (v, \ul{w}_1) \gtrdot \cdots \gtrdot  (v, \ul{w}_n) \gtrdot   \hat{0}$ is maximal in $[\hat{0}, (v,\ul{w})]$.

	It remains to show uniqueness. By construction, any increasing maximal chain in $[\hat{0}, (v,\ul{w})]$ can be extended to an increasing maximal chain in $[\tilde{0}^\wedge, (i(v), th(\ul{w}))] \subset \tilde{Q}_2^\wedge$ under the map $f$. The uniqueness follows from the corresponding uniqueness statement in $[\tilde{0}^\wedge, (i(v), th(\ul{w}))]$. 
	 
	 (a) is proved.

	By construction and \S\ref{subsec:LW}, for any $(v, \ul{w})  \gtrdot \alpha \ge \hat{0}$ with $\alpha \neq (v_1, \ul{w}_1)  $ in $\hat{Q}_2$, we have $\lambda \Big((v, \ul{w})  \gtrdot(v_1, \ul{w}_1)   \Big) < \lambda \Big((v, \ul{w})  \gtrdot \alpha   \Big)$.  By Proposition~\ref{prop:EL}, the labeling we constructed is an EL-labeling.

 
\section{Applications}

\subsection{Braid varieties}\label{subsec:braid}
Let $\ul{s} = (s_{i_1}, \dots, s_{i_n}) \in W^n$ be such that all $s_{i_j}$ are simple reflections. Let $w = m_\ast(\ul{s})$. We define the open braid variety associated to $\ul{s}$ as 
 \[
  \{ 
 (g_1, g_2, \dots, g_n)  \in B^+ \dot s_{i_1} B^+ \times^{B^+}  \cdots \times^{B^+} B^+ \dot s_{i_n} B^+/B^+\vert 
 g_1g_2\cdots g_n /B^+ = \dot{w} B^+/ B^+
 \}.
 \]
 Note that $B^+ \dot s_{i_1} B^+  \cdots   B^+ \dot s_{i_n} B^+/B^+ \subset \overline{B^+ \dot{w} B^+/B^+}$ and  $\overline{B^+ \dot{w} B^+/B^+} \cap B^- \dot{w} B^+/B^+ = \dot{w} B^+/B^+$. So the open braid variety is exactly $\mathring{\CZ}_{n,w, \ul{s}}$. Hence the closure is ${\CZ}_{n,w, \ul{s}}$.

We have the following regularity result for the braid varieties. 

\begin{thm}\label{thm:braid}
Let $\ul{s} = (s_{i_1}, \dots, s_{i_n})$. The totally nonnegative braid variety $\CZ_{n,m_{\ast}(\ul{s}), \ul{s}, \ge 0}$ is a regular CW complex homeomorphic to a closed ball.
\end{thm}

\begin{remark}
This theorem provides a geometric realization of the subword complex \cite[Corollary~3.8]{KM}.
	
On the other hand, the cluster algebra structure on the coordinate ring of the open braid variety $\mathring{\CZ}_{w, \ul{s}}$ has recently been established in \cite{CGG, GLBS} for a reductive group $G$. It is interesting to compare the total positivity structure with the cluster algebra structure.
\end{remark}

\begin{proof}
Let $x =  (m_{\ast}(\ul{s}), \ul{s}) \in Q_n$. The poset $[\hat{0},x]$ is the dual face poset of the subword complex considered in \cite[\S2]{KM}. Since a poset is shellable if and only if the dual poset is shellable, $[\hat{0},x]$ is shellable by \cite[Theorem~2.5]{KM}. Now the theorem follows from Proposition~\ref{thm:CW}. 
\end{proof}
  
\subsection{Double flag varieties}\label{subsec:Lusztig} Let $\CZ = \CZ_2 = G \times^{B^+} G/B^+$ in this subsection. 

\subsubsection{} We have the natural $G \times G$-action on $\CB \times \CB$. By \S \ref{sec:flag}, we have the decomposition of $\CB \times \CB$ into the $B^+ \times B^-$-orbits 
\begin{equation}\label{eq:BB}
\CB \times \CB=\sqcup_{v, w \in W} \mathring{\CB}_v \times \mathring{\CB}^w.
\end{equation} 
 
Let $G_\D=\{(g, g); g \in G\} \subset G \times G$. Following \cite{WY}, for $(u, v, w) \in W^3$, we consider the intersection of the $G_\D$-orbit  and the $B^+ \times B^-$-orbit $$\mathring{\CX}^u_{v, w}=(G_{\D} \cdot (B^+/B^+, \dot u B^+/B^+)) \cap  (\mathring{\CB}_v \times \mathring{\CB}^w).
$$
We have a decomposition of $ \CB \times \CB$ as $\sqcup_{u, v, w \in W} \mathring{\CX}^u_{v, w}$. We denote this stratified space by $\CX$ to emphasize that it is different from the stratified space in \eqref{eq:BB}.

The isomorphism $G \times G \to G \times G$ defined by $(g_1, g_2) \mapsto (g_1, g_1 g_2)$ induces an isomorphism of stratified spaces
   \begin{equation}\label{eq;alpha}
   \a:  \CZ  \to \CX, \qquad (g_1, g_2B^+ /B^+) \mapsto (g_1 B^+ /B^+, g_1 g_2 B^+ /B^+),
  \end{equation} 
mapping $\mathring{\CZ}_{v,\ul{w} }$ isomorphically to $\mathring{\CX}_{w_1, v}^{w_{2}}$. We define $\CX_{\ge 0} = \alpha(\CZ_{\ge 0})$ and $\CX^{w_2  }_{w_1, v  , >0} = \CX^{w_2  }_{w_1, v} \cap \CX_{\ge 0} =  \alpha(\CZ_{v, (w_1, w_2) ,>0})$. Combining Proposition~\ref{thm:CW} with Theorem~\ref{thm:conj}, we obtain the following regularity result for the double flag varieties. 

\begin{thm}\label{thm:double}
Let $v \in W$ and $(w_1, w_2) \in W^2$ with $v \le w_1 \ast w_2$. Then the closure of $\CX^{w_2  }_{w_1, v, >0}$ is a regular CW complex homeomorphic to a closed ball. 
\end{thm}

\subsubsection{}Since $\mathring{\CX}^u_{v, w}$ is irreducible by Corollary~\ref{cor:irre} and \eqref{eq;alpha}, there is a  unique $B^- \times B^+$-orbit $\mathring{\CB}^{v'} \times \mathring{\CB}_{w'}$  such that $(\mathring{\CB}^{v'} \times \mathring{\CB}_{w'})\cap \mathring{\CX}^u_{v, w}$ is dense in $\mathring{\CX}^u_{v, w}$. By \cite[\S1]{He-min}, for any $x,y \in W$, the subset $\{yu \in W\vert  u \le x\}$ has a unique minimal element. We denote the unique minimal element by $y \circ_r x$. 
 
\begin{prop}\label{lem:generic}
Let $v' = w \circ_r u^{-1} $ and $w'=  v \ast u$. Then 

 (1) $\CX^{u  }_{v,  w  , >0} \subset \mathring{\CX}^u_{v, w}\cap  (\mathring{\CB}^{v'} \times \mathring{\CB}_{w'}) $;
 
 (2) $\mathring{\CX}^u_{v, w}\cap  (\mathring{\CB}^{v'} \times \mathring{\CB}_{w'}) $ is dense in $\mathring{\CX}^u_{v, w}$.
\end{prop}

\begin{remark}
Part (1) arises from a question that Lusztig asked about the totally positive big cells in the double flag for reductive groups. 
\end{remark}

\begin{proof}
Let $\ul{w}= (w_1, w_2)$ be such that $w =w _1 w_2$ and $\ul{w}$ is a positive subexpression in $(v,u)$. Let ${\mathbf v}$ (resp. ${\mathbf u}$) be reduced expressions of $v$  (resp. $ u$). Let $({\mathbf w_1})_+$ (resp. $({\mathbf w_2})_+$) be the positive subexpression of $w_1$ (resp. $w_2$) in ${\mathbf v}$ (resp. ${\mathbf u}$). By \S\ref{subsec:posi} (b), we have $v' =w_1$. 
By Theorem~\ref{thm:Zp}, we have the isomorphism $$
    G_{(\mathbf w_1)_+, \mathbf v, >0} \times  G_{(\mathbf w_2)_+, \mathbf u, >0} \to \CX^{u  }_{v,  w  , >0}, \quad (g_1, g_2) \mapsto (g_1B^+/B^+, g_1g_2 B^+/B^+).
    $$
This proves (1). 
 
We now prove (2). Let $a, b \in W$ be such that $\mathring{\CX}^u_{v, w}\cap  (\mathring{\CB}^{a} \times \mathring{\CB}_{b})  \neq \emptyset$. Due to (1), it suffices to show $v' \le a$ and $b \le w'$. 

We first have 
\[
G_{\Delta} \cdot (B^+/B^+, \dot{u} B^+/B^+) \cap (\mathring{\CB}^{a} \times \mathring{\CB}^{w}) \neq \emptyset.
\]
Hence $( B^- \dot{w} B^+/B^+) \cap( \dot{a} B^+\dot{u} B^+/B^+)  \neq \emptyset$. Note that $
\dot{a} B^+\dot{u} B^+/B^+ \subset \sqcup_{x \le a \ast u} B^+ x B^+/B^+$. Hence $w \le a \ast u$,  or equivalently $w \circ_r u^{-1} \le a$ by a proof similar to \cite[Lemma 4.3]{HL}. 

Similarly, we have 
\[
G_{\Delta} \cdot (B^+/B^+, \dot{u} B^+/B^+) \cap (\mathring{\CB}_{v} \times \mathring{\CB}_{b}) \neq \emptyset.
\]
Hence $ (B^+\dot{b} B^+/B^+) \cap   (\dot{v} B^+ \dot{u} B^+/B^+) \neq \emptyset$. Hence $ b \le v \ast u$.
\end{proof}

\subsubsection{} We assume that $G$ is a connected reductive group.  Let $w_0$ be the longest element of the finite Weyl group $W$. 
By \cite[Theorem~8.7]{Lus-1}, we have $U^-_{>0} B^+/B^+=U^+_{>0}\dot{w}_0B^+/B^+$. Thus 
$$
\alpha(U^-_{>0}, U^-_{>0}B^+/B^+)=\alpha(U^-_{>0}, U^+_{>0}\dot{w}_0B^+/B^+)=(G_{>0})_{\D} \cdot (B^+/B^+, \dot{w}_0B^+/B^+).
$$

We may naturally identify $\CX$ with $G/ B^+ \times G/B^-$.   Then $\CX_{\ge 0}$ is identified with the Hausdorff closure of $(G_{> 0})_{\Delta} \cdot (B^+/B^+, B^-/B^-)$ in $G/ B^+ \times G/B^-$. The latter space was considered in \cite{WY}.

As a special case of Theorem~\ref{thm:Zp}, we have that $\CX_{\ge 0}$ is a union of semi-algebraical cells. This verifies a conjecture of Webster and Yakimov \cite[Conjecture 1]{WY}.

\subsection{Links on double Bruhat cells} We still write $\CZ =\CZ_2$ in this subsection. 

\subsubsection{}

We define the link of the cell $\CZ_{v', \ul{w}', >0} $ in the space $\CZ_{v,\ul{w}, \ge 0}$  for any $(v', \ul{w}') < (v, \ul{w}) \in Q$. Our definition of links is similar to \cite[\S3.1]{GKL}. 
 
 We  first define a closed subspace of $\CZ_{\ge 0}$ as follows
 $$
 {\rm Star}_{\ge (v', \ul{w}')} = \bigsqcup_{ (v', \ul{w}') \le (v'', \ul{w}'') \in Q} \CZ_{(v'', \ul{w}''), >0}.
 $$
  Let $(u, \ul{u_1}) \in Q$ be a minimal element such that $ (u, \ul{u_1}) \le (v', \ul{w}')$. By Proposition~\ref{prop:TM}, we have the following isomorphism via restriction
  \[
\phi:   {\rm Star}_{\ge (v', \ul{w}')} \rightarrow \bigsqcup_{v'' \le v'} \tilde{\CB}_{i(v''), i(u), > 0}  \times  \bigsqcup_{\ul{w}' \le \ul{w}''}  \tilde{\CB}_{th(\ul{u_1}), th(\ul{w}''), >0}
   \]
 Applying  \eqref{eq:prod} again, we obtain the isomorphism 
 \begin{equation}\label{eq:lk1}
  {\rm Star}_{\ge (i(v'), \ul{w}')}\rightarrow \tilde{\CB}_{i(v'), i(u), >0} \times \tilde{\CB}_{th(\ul{u_1}), th(\ul{w}'), >0} \times  \bigsqcup_{v'' \le v'} \tilde{\CB}_{i(v''), i(v'), > 0}  \times  \bigsqcup_{\ul{w}' \le \ul{w}''}  \tilde{\CB}_{th(\ul{w}'), th(\ul{w}''), >0}.
\end{equation}

We recall some representation theoretical notations. Denote by $ \tilde{X}^{++}$ the set of regular dominant weights of $\tilde{G}$. For $\lambda \in \tilde{X}^{++}$, we denote by $V^\lambda$ the highest weight simple $\tilde{G}$-module over $\BC$ with highest weight $\lambda$. We denote the highest weight vector by $\eta_\lambda$. We denote by $V^\lambda(\BR)$ the $\BR$-subspace of $V^\lambda$ spanned by the canonical basis. We equip $V^{\lambda,\lambda}(\BR) = V^\lambda(\BR) \otimes_\BR V^\lambda(\BR)$ with the standard Euclidean norm with respect to the tensor product of the canonical bases, denoted by $|| \cdot ||$.
 
 We continue from \eqref{eq:lk1}. We consider the projection 
\begin{align}\label{eq:lk2}
 &\tilde{\CB}_{i(v'), i(u), >0} \times \tilde{\CB}_{th(\ul{u_1}), th(\ul{w}'), >0} \times  \bigsqcup_{v'' \le v'} \tilde{\CB}_{i(v''), i(v'), > 0}  \times  \bigsqcup_{\ul{w}' \le \ul{w}''}  \tilde{\CB}_{th(\ul{w}'), th(\ul{w}''), >0} \\
  \notag \rightarrow &
\bigsqcup_{v'' \le v'} \tilde{\CB}_{i(v''), i(v'), > 0}  \times  \bigsqcup_{\ul{w}' \le \ul{w}''}  \tilde{\CB}_{th(\ul{w}'), th(\ul{w}''), >0}.
\end{align}  
For any $\tilde{v}, \tilde{w} \in \tilde{W}$, we have a natural embedding 
 \begin{align}
\label{eq:lk3} 
\pi':  \mathring{\tilde{\CB}}^{\tilde{w}}(\BR) \times   \mathring{\tilde{\CB}}_{\tilde{v}}(\BR) &\xrightarrow{((\dot{\tilde{w}})^{-1}, (\dot{\tilde{v}})^{-1}) \cdot -} \tilde{U}^-(\BR)\tilde{B}^+/ \tilde{B}^+ \times \tilde{U}^-(\BR) \tilde{B}^+/ \tilde{B}^+ \\
  \notag &\xrightarrow{(u_1,u_2) \mapsto u_1\eta_\lambda \otimes u_2 \eta_\lambda} V^{\lambda,\lambda}(\BR) .
 \end{align}
 
 We take $\tilde{v} = i(v')$ and $\tilde{w} = th(\ul{w'})$ in \eqref{eq:lk3}, and denote the composition of \eqref{eq:lk1},  \eqref{eq:lk2} and  \eqref{eq:lk3} by  $\pi: {\rm Star}_{\ge (v', \ul{w}')} \rightarrow V^{\lambda,\lambda}(\BR)$.
 
For any $(v', \ul{w}') < (v'', \ul{w}'') \le (v, \ul{w}) \in Q$,  we define the link space as
\begin{align*}
 Lk_{(v', \ul{w}')} (\CZ_{v,\ul{w}, \ge 0}) &=\pi ( {\rm Star}_{\ge (v', \ul{w}')} \cap \CZ_{v,\ul{w}, \ge 0}) \bigcap \{  x \in V^{\lambda,\lambda}(\BR) \vert\,  ||x - \eta_\lambda \otimes \eta_\lambda || =1  \};\\
  Lk_{(v', \ul{w}'), (v'', \ul{w}''), > 0} &=  Lk_{(v', \ul{w}')} (\CZ_{v,\ul{w}, \ge 0})  \cap \pi ( {\rm Star}_{\ge (v', \ul{w}')} \cap \CZ_{v'',\ul{w}'', > 0}).
\end{align*}

We hence have a stratified space 
\[
 Lk_{(v', \ul{w}')} (\CZ_{v,\ul{w}, \ge 0}) = \bigsqcup_{(v', \ul{w}') < (v'', \ul{w}'') \le (v, \ul{w}) \in Q}   Lk_{(v', \ul{w}'), (v'', \ul{w}''), > 0}.
\] 
 It follows from the construction that the definition of $Lk_{(v', \ul{w}')} (\CZ_{v,\ul{w}, \ge 0})$ is independent of the choice of the minimal element $(u, \ul{u_1}) \in Q$ and the regular dominant weight $\lambda$, up to a stratified isomorphism. 

\begin{prop}\label{prop:lk}
Let  $(v', \ul{w}') < (v, \ul{w}) \in Q$. Then  $Lk_{(v', \ul{w}')} (\CZ_{v,\ul{w}, \ge 0})$  is a regular CW complex homeomorphic to a closed ball.
\end{prop}

\begin{proof}Let 
\begin{align*}
&S =  \{x \in V^{\lambda,\lambda}(\BR)  \vert\,  ||x - \eta_\lambda \otimes \eta_\lambda || =1\}, \quad R =  \tilde{\CB}^{i(v)}_{ \ge0} \times \tilde{\CB}_{th(\ul{w}), \ge 0} \subset \tilde{\CB}_{\ge 0} \times \tilde{\CB}_{\ge 0},\\
&T' =\bigsqcup_{v'' \le v'} \tilde{\CB}_{i(v''), i(v'), > 0}  \times  \bigsqcup_{\ul{w}' \le \ul{w}''}  \tilde{\CB}_{th(\ul{w}'), th(\ul{w}''), >0} \subset \tilde{\CB}_{\ge 0} \times \tilde{\CB}_{\ge 0} , \quad T = T' \cap R.
\end{align*}

It follows from the construction that $Lk_{(v', \ul{w}'), (v, \ul{w}), \ge 0}$ is isomorphic to $\pi'(  T ) \cap  S$ as stratified spaces.  The space $\pi'(  T ) \cap  S$ is a special case of links considered in \cite[\S4]{BH22} related to the flag variety of the Kac-Moody group $\tilde{G} \times \tilde{G}$. 

By \cite[Proposition~4.4]{BH22}, $\pi'(  T ) \cap  S$ is a regular CW complex homeomorphic to a closed ball. The proposition now follows.
\end{proof}
 
\subsubsection{} We assume that $G$ is a connected reductive group. We denote by $w_0$ the longest element in the Weyl group of $G$.  

  For any $v, w \in W$, we define the double Bruhat cell $ G^{v,w} = B^+ \dot{w} B^+ \cap B^- \dot{v} B^-$. Let $\pi: G \rightarrow G/T = L$ be the quotient map.  We define the reduced double Bruhat cells as $L^{v,w} = \pi(G^{v,w})$. By \cite[Proposition~2.1]{WY}, we have the stratified embedding 
  \[
	f :  L \rightarrow \CB \times \CB \cong \CZ, \quad gT \mapsto (gB^+/ B^+, g\dot{w}_0B^+/ B^+) \mapsto (g, \dot{w}_0B^+/ B^+),
\]
mapping $L^{u,v}$ isomorphically to $\mathring{\CZ}_{{ww_0}, (v, w_0)}$  We further define $L_{\ge 0} = \pi(G_{\ge 0})$. It is clear that $L_{\ge 0} = \sqcup_{v,w} L^{v,w}_{ >0}$ with $L^{v,w}_{ >0} = L^{v,w} \cap L_{\ge 0}$. 

Let ${\mathbf v} = (s_{i_1}, \cdots, s_{i_m} )$ and ${\mathbf w} = (s_{j_1}, \cdots, s_{j_n})$ be reduced expressions of $v$ and $w$, respectively.  It follows by \cite{Lus-1} that the map 
\begin{align}\label{eq:db}
	\BR_{> 0}^{m+ n } &\longrightarrow L^{v,w}_{ >0}, \\
	\notag  (a_1,\dots, a_n, b_1, \dots, b_m) &\mapsto y_{j_1} (a_1) \cdots y_{j_n} (a_n) x_{i_1}(b_1) \cdots x_{i_m}(b_m),
\end{align}
is an isomorphism. 
\begin{lem}
We have $L_{\ge 0} =\overline{\pi(G_{>0})}$ in $L$. 
\end{lem}

\begin{proof}
Since $\overline{G_{>0}} = G_{\ge 0}$, we have $L_{\ge 0} \subset \overline{\pi(G_{>0})}$. 
We have the following commutative diagram 
\[
\begin{tikzcd}
L_{\ge 0} \ar[d] \ar[r,equal] & \sqcup_{v,w}L^{v,w}_{ >0} \ar[r, hookleftarrow] \ar[d] & L^{v,w}_{ >0} \ar[d]\\
\overline{\pi(G_{>0})} \ar[r] &  \overline{ \CZ_{e, (w_0, w_0), >0}} \cap L = \sqcup_{v,w}\CZ_{vw_0, (w, w_0), >0} \ar[r, hookleftarrow]& \CZ_{vw_0, (w, w_0), >0}.
\end{tikzcd}
\]

We claim $L^{v,w}_{ >0} \rightarrow \CZ_{vw_0, (w, w_0), >0}$ is surjective, hence an isomorphism. Let ${\mathbf v} = (s_{i_1}, \cdots, s_{i_m} )$ and ${\mathbf w} = (s_{j_1}, \cdots, s_{j_n})$ be reduced expressions of $v$ and $w$, respectively. We consider the morphism following \eqref{eq:db}
\[
\BR_{> 0}^{m+ n } \longrightarrow L^{v,w}_{ >0}  \longrightarrow \CZ_{vw_0, (w, w_0), >0}
 \]
mapping $ (a_1,\dots, a_n, b_1, \dots, b_m)$ to  $(y_{j_1} (a_1) \cdots y_{j_n} (a_n), x_{i_1}(b_1) \cdots x_{i_m}(b_m) \dot{w}_0 B^+ /B^+)$. By \cite[Theorem~8.7]{Lus-1}, we have $ U^+_{>0} \dot{w}_0 B^+/ B^+ = U^-_{>0} B^+/ B^+$. Hence the image of $\BR_{> 0}^{m+ n }$ is the entire $\CZ_{vw_0, (w, w_0), >0}$ by Theorem~\ref{thm:Zp}. This finishes the lemma.
\end{proof}

Via the embedding $L_{\ge 0} \rightarrow \CZ_{\ge 0}$, we can define the link of $L^{v,w}_{ >0}$ in $L_{\ge 0}$ as 
\[
 Lk_{(v,w)}(L_{\ge 0}) = Lk_{(v, (w,w_0))} (\CZ_{\ge 0}), \quad \text{for any } v, w \in W \text{ such that } (v,w) \neq (w_0, w_0).
\]
 By Proposition~\ref{prop:lk}, we have the following theorem, solving an open problem of Fomin and Zelevinsky.
\begin{thm}\label{thm:link}
Let $v, w \in W$ such that $(v,w) \neq (w_0, w_0)$. The link $Lk_{(v,w)}(L_{\ge 0})$ of $L^{v,w}_{ >0}$ in $L_{\ge 0}$   is a regular CW complex homeomorphic to a closed ball.
\end{thm}


\end{document}